\documentclass[english,colorinlistoftodos]{scrartcl}

\usepackage{amsmath, amsthm, amssymb,  mathtools}
\usepackage[utf8]{inputenc}
\usepackage{multicol}
\usepackage{tikz}
\usepackage{fancyvrb}
\usepackage{hyperref} 
\usepackage{cleveref}[capitalise]	
\usepackage[inline]{enumitem}
\usepackage{authblk}
\usepackage{caption}
\usepackage{subcaption}
\usepackage{multirow}
\usepackage{bm, bbm}
\usepackage{algorithm}
\usepackage{algpseudocode}
\usepackage[space]{grffile}
\usepackage{epigraph}
\usepackage{float}

\usepackage{epsfig}
\usepackage{graphicx}

\graphicspath{{Figures/}} 

\usepackage{xcolor}

\newcommand\cB{\mathcal{B}}
\newcommand\cC{\mathcal{C}}

\newcommand\cE{\mathcal{E}}
\newcommand\cF{\mathcal{F}}
\newcommand\cG{\mathcal{G}}

\newcommand\cL{\mathcal{L}}
\newcommand\cM{\mathcal{M}}

\newcommand\cP{\mathcal{P}}

\newcommand\cU{\mathcal{U}}

\newcommand\cW{\mathcal{W}}
\newcommand\cX{\mathcal{X}}
\newcommand\cY{\mathcal{Y}}

\newcommand\bE{\mathbb{E}}

\newcommand\bN{\mathbb{N}}

\newcommand\bP{\mathbb{P}}

\newcommand\bR{\mathbb{R}}

\newcommand\pp[1]{\left( #1 \right)}
\newcommand\gp[1]{\left\lbrace #1 \right\rbrace}
\newcommand\sq[1]{\left[ #1 \right]}
\newcommand\abs[1]{\left| #1 \right|}
\newcommand{\norm}[1]{\left\Vert #1\right\Vert}

\newtheorem{theorem}{Theorem}
\newtheorem{corollary}[theorem]{Corollary}
\newtheorem{lemma}[theorem]{Lemma}
\newtheorem{proposition}[theorem]{Proposition}
\theoremstyle{definition}
\newtheorem{definition}[theorem]{Definition}
\newtheorem{example}[theorem]{Example}
\newtheorem{remark}[theorem]{Remark}

\newcommand{\leqcx}{\leq_{\text{cx}}}

\hypersetup{colorlinks = true, urlcolor = blue, citecolor = blue, filecolor=blue, linkcolor=blue}
\begin{document}
\title{How to quantise probabilities while preserving their convex order}
\medskip

\author{\scshape{Marco Massa\thanks{\href{mailto:m.massa17@imperial.ac.uk}{m.massa17@imperial.ac.uk}} 
\;and\; 
Pietro Siorpaes\thanks{\href{mailto:p.siorpaes@imperial.ac.uk}{p.siorpaes@imperial.ac.uk}}}\\
{\small{\textit{Dept.~of Mathematics, Imperial College London, 16-18 Princess Gardens, SW71NE}
\\ 
\small{\textit{London, UK}}}}
}

\maketitle

% ABSTRACT - START
\vspace*{-2cm}
\section*{
\begin{center}
	\small
	Abstract
	\end{center}
}
\vspace*{-.6cm}
\hfil
\begin{minipage}[c]{0.8\textwidth}
\small
We introduce an algorithm which, given probabilities $\mu  \leqcx \nu$ in convex order and defined on a separable Banach space $B$, constructs finitely-supported approximations $\mu_n \to \mu, \nu_n\to \nu$ which are in convex order   $\mu_n \leqcx \nu_n$. We provide upper-bounds for the speed of convergence, in terms of the Wasserstein distance. 
We discuss the (dis)advantages of our algorithm and its link with the discretisation of the Martingale Optimal Transport problem, and we illustrate its  implementation with numerical examples. 
We study the operation which, given $\mu$/$\nu$ and some (finite) partition of $B$, outputs $\mu_n$/$\nu_n$, showing that applied to a probability $\gamma$ and to all partitions it outputs the set of all probabilities $\zeta\leqcx \gamma$.
\end{minipage}
\hfil
% ABSTRACT - END

\medskip

\noindent \textbf{MSC2020 subject classifications:} 
Primary 49M25, 60E15; secondary 60-08, 90C08.\\
\noindent \textbf{Keywords:} \textit{Convex order, quantisation, barycentric quantisation, martingale optimal transport, Wasserstein distance.}

\section{Introduction}
\label{Section: Introduction}
The main contribution of this paper is the introduction of an algorithm which, given probabilities $\mu  \leqcx \nu$ in convex order and with finite first moment, defined on a separable Banach space $B$, constructs finitely-supported approximations $\mu_n,\nu_n$ which are in convex order   $\mu_n \leqcx \nu_n$ and which converge to $\mu,\nu$ in Wasserstein distance.  

Obtaining a quantisation algorithm which preserves the convex order is important because of the link with the   MOT (\emph{Martingale Optimal Transport}) problem, a version of the OT problem in which one minimises not over  the set $\Pi(\mu,\nu)$ of all transports $\pi$ from $\mu$ to $\nu$, but rather over the  subset $\cM(\mu,\nu)$ of those transports which additionally satisfy the martingale constraint, i.e.~which are the joint laws of $(X,Y)$ such that 
$$X\sim \mu, \quad Y \sim \nu, \quad \bE[Y|X]=X.$$ The link is provided by Strassen's theorem, which states that $\cM(\mu,\nu)$  is non-empty if and only if $\mu  \leqcx \nu$.  It is important to quantise the probabilities since, in the case of finitely-supported measures, the MOT problem is a linear program, and is thus easily solved numerically. 

\medskip

This problem has already been considered by several authors, all of whom consider probabilities in the space $\cP_p(B)$ of measures on $B$ with finite $p^{th}$ moment (where $p\in [1,\infty)$), endowed with the $p$-Wasserstein distance $\cW_p$, in the setting of finite-dimensional $B$.
 
The first result, due to Baker \cite{Baker2015}, considers the case where the measures are defined on the real line $B=\bR$. In this case, Baker provides an algorithm which, given $\mu,\nu\in \cP_1$ and $n\in \bN$, returns a measure $\mu_n\eqqcolon \cU_n(\mu)$ supported on at most $n$ points and such that  $\cU_n(\mu)\to \mu$ in $\cP_1$, and $\mu  \leqcx \nu$ implies $\cU_n(\mu)  \leqcx \cU_n(\nu)$. While Baker's $\cU$-quantisation $\cU_n(\mu)$ provides a great solution to the problem at hand,  it is only defined in dimension one, and it turns out that finding a quantisation which preserves the convex order, just like most questions about MOT, is a lot more challenging in dimension higher than $1$; indeed, as remarked by   \cite{Beiglbock2022} 
	\textit{`it is a highly intriguing challenge to extend the martingale transport theory to the case where $\mu, \nu$ are supported on $\bR^d, d > 1$'}.

The \emph{dual quantisation} proposed in \cite{Pages2012},  though defined for general $B=\bR^N$, preserves the convex order only in dimension $N=1$ \cite{Alfonsi2019}. A variant is considered in \cite{JourdainPages2022}, who applied the optimal quadratic quantisation to the first marginal $\mu$, and the dual quantisation  to the second marginal $\nu$. This scheme does preserve the convex order in any dimension; however, both of the algorithms in \cite{Pages2012, JourdainPages2022}  are only defined for probabilities $\mu,\nu$ with compact support, and  do not generalise to the case of several (not just two) marginals. 

Another two discretisation techniques  have been considered in  \cite{Alfonsi2019IJTAF, Alfonsi2019}.  Given $\mu \leqcx \nu$   and \emph{arbitrary}  finitely-supported probabilities $(\mu_n,\nu_n)$ converging to $(\mu, \nu)$ in $\cP_p$,   the metric projection $\alpha_n$ of $\mu_n$ on $\{\alpha: \alpha \leqcx \nu_n\}$ exists and it satisfies $\alpha_n \to \mu$ in $\cP_p$ (and, trivially, $\alpha_n \leqcx \nu_n$ and $\alpha_n $ is finitely-supported). Analogously,  the metric projection $\beta_n$ of $\nu_n$ on $\{\beta: \mu_n \leqcx   \beta\}$ exists and satisfies $\beta_n \to \nu$ in $\cP_p$ (and, trivially, $\mu_n \leqcx \beta_n$ and $\beta_n $ is finitely-supported).
While this method is conceptually very neat, it does have the disadvantage that $\alpha_n,  \beta_n$  cannot be explicitly computed, though this issue admits the following  partial workaround. If  $\mu_n$ is the empirical measure $\frac{1}{n}\sum_{i=1}^n \delta_{X_i}$  corresponding to IID random variables $X_i\sim \mu$ and analogously for  $\nu_n$,  then  the corresponding  random measure $\alpha_n$  solves a quadratic optimisation problem with linear constraints, and so it can be computed numerically (not so for $\beta_n$) and $(\alpha_n, \nu_n) \to (\mu, \nu)$ a.s..

An entirely different approach is investigated by \cite{Guo2019}. Their numerical scheme relies on an \emph{arbitrary} discretisation of the measures $\mu,\nu$,  along with an appropriate \emph{relaxation} of the martingale constraint, which only requires that
\begin{equation}
\label{relaxed martingale condition}
	\bE\sq{\big\vert\bE\sq{Y \vert X}-X\big\vert} \leq \epsilon ,
\end{equation} 
holds for a $\epsilon\geq 0$, where $X,Y$ have laws $\mu,\nu$; the case $\epsilon=0$ corresponds to the martingale constraint. 

The algorithm which we propose constructs $\mu_n$  in a way which generalises  Baker's $\cU$-quantisation (though this link is not obvious), and it is what we call a \emph{proper barycentric} quantisation. By this, we mean that $\mu_n$, which being finitely-supported  is of the form 
$$
\mu_n  = 
\sum_{i=1}^n \alpha_i^{}\delta_{x^{}_i} ,
\quad \text{ where } 
\quad  
\alpha_i> 0 \quad  \forall i, 
\quad 
\sum_{i=1}^n\alpha^{}_i =1, 
$$
satisfies $\mu_n \leqcx \mu$, and can be calculated from the knowledge of $\mu$ alone, in the sense that there exist a finite partition  $\Pi^n=\gp{B_1,\ldots, B_n}$ of $B$, made of Borel sets, such that $\alpha_i=\mu(B_i)$, and $x_i$ is the barycentre of the probability $\mu_i\coloneqq \frac{\mu( \cdot \cap B_i)}{\mu(B_i)}$, for all $i$. 

Our construction builds instead $\nu_n$ as a (not necessarily proper) \emph{barycentric} quantisation of $\nu$ (see Definition \ref{Definition: barycentric quantisation measures});  it turns out that this more general construction, unlike the proper quantisation, allows to obtain all (and only) the finitely-supported probabilities $\nu_n \leqcx \nu$. To build a barycentric quantisation $\nu_n$ of $\nu$  requires the knowledge not just of $\nu$, but rather of a martingale transport $\pi$ from $\mu$ to $\nu$, and the choice of two partitions $\Pi_1^n=(P_{1,i}^n)_i,\Pi_2^n=(P_{2,j}^n)_j$ of $B$; if $\Pi_1^n$ is chosen to be equal to  the partition $\Pi^n$ used to build $\mu_n$, and $\mu \leqcx \nu$, then  $\mu_n\leqcx  \nu_n$. 

If $(\Pi_1^n)_n, (\Pi_2^n)_n$ are refining  then $(\mu_n)_n$ and $(\nu_n)_n$ are increasing in convex order, and if moreover $\vee_n\sigma(\Pi_1^n)$ and $\vee_n\sigma(\Pi_2^n)$ equal the Borel $\sigma$-algebra $\cB(B)$ of $B$ then the corresponding $(\mu_n)_n, (\nu_n)_n$ converge to $\mu, \nu$. We also prove the upper-bounds 
$$
\cW_p(\mu_n,\mu)\leq 	\sum_i (\text{diam}\pp{P_{1,i}})^p \mu\pp{P_{1,i}} , \quad \cW_p(\nu_n,\nu)\leq \sum_{j} (\text{diam}\pp{P_{2,j}})^p \nu\pp{P_{2,j}} .
$$
 Needing a $\pi\in \cM(\mu,\nu)$   is a con of our approach, because only $\mu,\nu$ (not $\pi$) are  inputs of the MOT problem  and while the existence of such a $\pi$ is guaranteed by Strassen's Theorem, algorithms which produce a $\pi\in \cM(\mu,\nu)$ given  arbitrary $\mu \leqcx \nu$ are so far only known in dimension one \cite{BeiglbockJuillet2016,Juillet2016}. We also mention the constructions reported in \cite[Section 4.2]{Labordere2017} e.g.~Bass's construction \cite{Bass1983} related to the Skorokhod embedding problem and the local variance gamma model of \cite{CarrNadtochiy2017} motivated by the application of MOT to obtain an arbitrage-free measure consistent with given market option prices. We leave to separate papers  the difficult task of developing  algorithms which work in the multi-dimensional setting.

On the other hand, our approach has many pros: it leads to explicit expressions for $\mu_n,\nu_n$ which can easily be calculated numerically by evaluating integrals, it outputs non-random $\mu_n \leqcx \nu_n$, it works when $B$ is an arbitrary separable Banach space and the proofs turn out to be easy,  as soon as one has the right idea. The idea is to consider an analogous statement for random variables, just like Skorokhod's representation Theorem allows to draw a parallel between weak convergence of measures and convergence in probability of random variables. This is made possible by the equivalent characterisation of the convex order provided by Strassen's Theorem.

\medskip

As we mentioned, finding a quantisation which preserves the convex order is of interest when one wants to numerically solve the MOT problem; however, this is only one piece of the puzzle, since 
to carry out this solution  it is necessary to identify conditions under which $\text{MOT}_n \to \text{MOT}$, i.e.,  the MOT problem with marginals $(\mu_n,\nu_n)$ converges to the MOT problem with marginals $(\mu,\nu)$ (in an appropriate sense) when 
\begin{align}
\label{eq: input MOT_n to MOT}
	\mu_n \to \mu, \quad  \nu_n \to \nu, \quad \mu_n \leqcx \nu_n, \quad \mu \leqcx \nu .
\end{align}
Solutions to this problem have been provided from different perspectives in \cite{Backhoff2019, Juillet2016, Wiesel2019}, who essentially showed that the MOT problem is stable, i.e., no conditions other than (\ref{eq: input MOT_n to MOT})  are required to obtain the convergence of the MOT problems. 
 However, all these papers consider only the one-dimensional setting and, as usual, the situation in the multi-dimensional setting is a lot more complicated; in particular, the MOT problem lacks stability in this setting \cite{BrJu22}. While we leave for a separate paper the work of finding conditions under which $\text{MOT}_n \to \text{MOT}$,  in this paper we do provide a  statement in this direction, which in particular shows that if the transport $\pi$ which one uses to build $(\nu_n)_n$ already happens to be optimal for the MOT problem (with inputs $(\mu,\nu)$), then $\text{MOT}_n \to \text{MOT}$. 
 Of course, this is a rather weak and not particularly useful statement, because the whole point of using $\text{MOT}_n$ to approximate $\text{MOT}$ is that one does not know a priori the optimiser of MOT. 
We also carry out numerical examples in which we illustrate the implementation of the proposed construction.
  
In the last section, we define in somewhat more abstract terms what is a barycentric quantisation and we study the properties of such construction. 
In particular, we show that the (proper) barycentric quantisations, which we define separately for measures and for random variables, are closely connected and that one can characterise which measures $\zeta$ such that $\zeta\leqcx \gamma$ are finitely-supported (resp.~supported on at most $n$ points) using barycentric quantisations of $\gamma$, but not using only \emph{proper} barycentric quantisations. 
	 
Finally, we recall that it is is natural to consider a more general OT problem, in which there are multiple marginals, rather than just two as described above. In this paper we concentrate on the case of two marginals, because the extension of our results to the case of finitely many marginals is trivial, though it complicates the notation.

\medskip

The rest of this article is structured as follows.
In Section \ref{Section: Optimal Transport problems}, we briefly introduce the OT problem and its several variants with particular focus on the MOT problem. 
In Section \ref{Section: discrete methods}, we discuss the discretisation technique for an OT problem with additional linear (or convex) constraints and state several questions which we will then address in this paper in the setting of MOT. 
The core of our paper is Section \ref{Section: discretisations preserving the convex order}, in which we introduce our quantisation algorithm which preserves the convex order, providing both an explicit representation, bounds on the speed of convergence and a remark about stability.
In Section \ref{Section: Numerical examples}, we provide several examples in which we show the implementation of the proposed martingale quantisation and the resolution of the corresponding discretised MOT problem.
Finally, in Section \ref{Section: Barycentric quantisation} we introduce  the notion of barycentric quantisation and study its properties.

\section{Optimal Transport problems}
\label{Section: Optimal Transport problems}
In this section, we recall the OT problem,  briefly mention several variants thereof (among which the MOT) and introduce the notation and setting used in this paper.

\subsection{The classic Optimal Transport problem}
\label{Section: The classic Optimal Transport problem}
 Consider probability measures $\mu, \nu$ on Polish spaces $\cX,\cY$ endowed with their Borel $\sigma$-algebras, and a non-negative measurable cost function $c:\cX\times \cY \longmapsto [0,\infty]$.  In the 1781 article  \cite{Monge},  Monge considered (a special case of) the OT problem 
\begin{equation}
\label{Monge problem}
	\inf_{T_\#\mu=\nu}\int_\cX c(x,T(x))d\mu(x) ,
\end{equation}
where the infimum runs over all  measurable function $T:\cX\rightarrow \cY$, called \emph{transport maps}, such that  $T_\#\mu\coloneqq\mu \circ  T^{-1}$ equals $\nu$. This constraint ensures that a feasible transport map is one that pushes forward $\mu$ into $\nu$, i.e.~it does not waste (or generate) mass during the transportation. This  formulation of the OT problem has drawbacks, mostly notably the non-linearity of the constraint makes it very complex.

In 1942, Kantorovich \cite{Kantorovich1, Kantorovich2} formulated the OT problem from a  different perspective, which has become the standard one. Instead of minimising over transport maps $T$, he used as variables \emph{transport plans}, defined as those probability measures $\pi\in \mathcal{P}\pp{\cX\times \cY}$ on the product space $\cX\times \cY$ (endowed with its Borel $\sigma$-algebra $\cB(\cX\times \cY)=\cB(\cX) \times \cB(\cY)$), which have marginals $\mu$ and $\nu$, i.e.~ 
\begin{equation}
\label{marginals mu nu}
	\pi\pp{\cdot\times \cY} = 
	\pi \circ P_x^{-1} = \mu 
	\quad, 
	\quad
	\pi\pp{\cX\times \cdot} = 
	\pi \circ P_y^{-1} = \nu.
\end{equation}
where $P_x$ and $P_y$ denote the projections on the first and second coordinate.

Denoting with  $\Pi\pp{\mu, \nu}$  the set of such transport plans, Kantorovich's formulation of OT problem can be stated as
\begin{equation}
\label{Kantorovich problem}
\tag{OT}
	P\pp{\mu, \nu}\coloneqq
	\inf_{\pi \in \Pi\pp{\mu, \nu}} \int_{\cX\times \cY} c(x, y)d\pi(x, y),
\end{equation}
or, in probabilistic jargon, as
$$
P\pp{\mu, \nu}\coloneqq
\inf_{X\sim \mu, Y\sim \nu}\bE\sq{c\pp{X, Y}}, 
$$
the infimum being over the set of (joint laws $\pi$ of) random variables $X, Y$ with laws $\cL(X) = \mu, \cL(Y) = \nu$. 
We always denote with $(\Omega, \cF, \bP)$ the underlying probability space, and $\bE$ denotes the expectation with respect to $\bP$. If we want to highlight the fact that $(X,Y)$ have law $\pi$, we write $\bE^{\pi}\sq{f\pp{X, Y}}$ for $\bE\sq{f\pp{X, Y}}$ (we use analogous notations for the conditional expectation).

Since the objective function of (\ref{Kantorovich problem}) is linear in $\pi$ and $\Pi\pp{\mu, \nu}$ is defined by linear inequalities, this formulation renders the OT problem a linear optimisation problem, thus making it much more tractable. Moreover, $\Pi\pp{\mu, \nu}$ is (non-empty\footnote{It always contains the product measure $\mu \times \nu$.} and) compact with respect to the weak topology of probability measures, and this ensures the existence of an optimal solution under very mild assumptions.

Clearly, the two formulations are closely related and the concept of transport plan introduced by Kantorovich extends the transport map, in the sense that it allows mass positioned at $x$ to be reallocated to multiple positions in $\cY$. In particular, if $\pi^*$ is an optimal transport plan and is of the form $\pi^*= ($Id$, T^*)_\#\mu $ for some map $T^*$, then $T^*$ is an optimal transport map.

OT theory permits to define a most useful distance on the space of probability measures. If $\cX=\cY$ is a Polish space equipped with a metric $d$, $p\in [1,\infty)$, consider the OT problem with cost function  $c\coloneqq d^p$, i.e.
\begin{equation}
\label{Wasserstein distance}
	\cW_p\pp{\mu, \nu} \coloneqq 
	\pp{ \inf_{\pi \in \Pi\pp{\mu, \nu}} \int_{\cX\times \cX} d\pp{x, y}^p d\pi(x, y)}^{1/p} .
\end{equation}  
The quantity $\cW_p$ defines a distance on the space\footnote{Though a priori the definition of $\mathcal{P}_p(\cX)$ seems to depend on the choice of $x_0\in \cX$, it is easy to show that it does not.} 
$$
\mathcal{P}_p(\cX)\coloneqq
\gp{\mu \in \cP(\cX): \int d(x_0,x)^p \mu(dx)<\infty} ,
$$ 
of probability measures on $\cX$ with finite moments of order $p$, often called the $p$\emph{-Wasserstein distance}. 

\begin{remark}
\label{Remark: equivalence convergence C0b C0p}
The following are equivalent
\begin{enumerate}
	\item $ \cW_p(\mu_n,\mu)\to 0$
	\item 	
		\begin{equation}
		\label{eq: convergence}
		\tag{$\star$}
			\int f d\mu_n \to \int f d\mu
		\end{equation}
		holds for all $f \in C^0_p(\cX)$ i.e~all continuous functions $f:\cX\to\bR$ for which there exists $C\in \bR$ such that
		$$
		\abs{f(x)} \leq C\pp{1+d(x,x_0)^p} 
		$$
 		for some (or equivalently, for any) $x_0 \in \cX$.
	\item 
		$\mu_n \to \mu$ \emph{weakly} (i.e.~(\ref{eq: convergence}) holds for $f:\cX \to \bR$ bounded and continuous) and (\ref{eq: convergence}) holds for $f(x) \coloneqq 1+d(x,x_0)^p$ for some (or, equivalently, for any) $x_0 \in \cX$.
\end{enumerate}
In particular, $\cW_1(\mu_n,\mu)\to 0$ if and only if $\int f d\mu_n \to \int f d\mu$ for every continuous $f:\cX \to \bR$  with sub-linear growth, and when $d$ is bounded, $\cW_p$ metrises the weak convergence. Finally,  $\mathcal{P}_p(\cX)$, metrised as always  by $\cW_p$, is complete and separable \cite[Theorem 6.18]{Villani2007}. 
\end{remark}

\subsection{Constrained OT}
\label{Section: Constrained OT}
Interest in OT has been growing over the past two decades due to its wide applications to different fields (e.g.~mathematical finance, economics, machine learning and image processing). It is then not surprising that many interesting variants of the original OT problem has been introduced and studied by many different authors. 
For example, in \emph{entropic} OT, one replaces the linear cost functional $\pi \mapsto \int c d\pi$ with a strictly convex one, which much improves the speed of convergence of numerical algorithms to solve the problem. 
Alternatively, one can keep the cost functional unchanged, but require the transports to satisfy additional constraints. In this case we talk of the \emph{constrained} OT problem (\ref{eq: cOT}), defined as 
\begin{equation}
	\label{eq: cOT}
	\tag{cOT}
	P_{\text{c}}\pp{\mu, \nu} \coloneqq \inf_{\pi \in \Pi_{\text{c}}\pp{\mu, \nu}} \int c d\pi ,
\end{equation}
thus substituting in  (\ref{Kantorovich problem}) the set $\Pi\pp{\mu, \nu}$ of transport plans by a set $\Pi_{\text{c}}\pp{\mu, \nu}$ of  transport plans which satisfy some additional constraints. As long as $\Pi_{\text{c}}\pp{\mu, \nu}$ is convex, (\ref{eq: cOT}) is a convex optimisation problem, and thus remains tractable. Several instances of (\ref{eq: cOT}) with linear constraints are found in the literature, e.g.~: 
\begin{enumerate}
\item \cite{Kormal14} considers capacity constraints, i.e.~those transports $\pi \ll \cL^{2N}$ whose density with respect to the Lebesgue measure $\cL^{2N}$ is bounded above by a fixed constant.
\item Working in a multi-marginal setting, \cite{Backhoff2017} consider \emph{causal} transports (or the closely-related bi-causal transports).  
These are those  joint laws $\pi$ of $(X, Y)$, with $X=\pp{X_{i}}_{i=1}^N, Y=\pp{Y_{i}}_{i=1}^N$ with values in $\bR^N$,  for which the conditional law of $Y_{1:t}\coloneqq\pp{Y_{1}, \ldots, Y_{t}}$ given $\pp{X_{1}, \ldots, X_{N}}$ is the same as the conditional law of $Y_{1:t}$ given $\pp{X_{1}, \ldots, X_{t}}$, for every $t=1, \ldots, N$.
\item Many authors considered the martingale constraint $ \bE^{\pi}\sq{Y\vert X}=X$, of which we talk in more detail in the next section.
\item  \cite{Zaev2015} considers the general  OT problem with linear constraints, and then discusses in more detail the two cases of transports which satisfy the martingale constraint, and of transports invariant under the  action of a (product) group.
\end{enumerate}
Unlike the classic OT problem, for which $\Pi\pp{\mu, \nu}$ is always non-empty, it can happen that  $\Pi_{\text{c}}\pp{\mu, \nu} = \emptyset$, in which case (\ref{eq: cOT}) is of no interest; this leads us to the following definition. 
\begin{definition}
	We say that $\pp{\mu, \nu}$ is \emph{viable} (for (\ref{eq: cOT})) if $\Pi_{\text{c}}\pp{\mu, \nu} \neq \emptyset$.
\end{definition}
It is of course of interest to characterise when $\pp{\mu, \nu}$ is viable; this depends on the specific nature of the constraints. In this regard, it is interesting to recall Tchakaloff's Theorem \cite{bayer2006proof}, which states that given $\mu\in \cP_1(B)$, if $B=\bR^d$ and $f^1, \ldots f^m\in L^1(\bP, B)$ represent the linear constraints $\int f^i d \mu=\int f^i d \hat{\mu}, i=1, \ldots, m$ to be verified by some looked-for finitely-supported $\hat{\mu}$, then such a quantisation $\hat{\mu}$ of $\mu$ always exists. Such theorem also admits a martingale version \cite{BeNu14}, and this is not at all obvious since the martingale constraint corresponds to infinitely many linear constraints. However, theese theorems only provide the existence of such quantisations, they do not describe a procedure to construct them explicitly.

\subsection{Martingale Optimal Transport}
\label{Section: MOT}
\indent The first formulation of the MOT problem appeared in \cite{Beiglbock2013} motivated by the \emph{model-independent pricing} problem in Mathematical Finance. The aim of this problem is to find lower and upper bounds for the price of an exotic pay-off resulting by all arbitrage-free  models consistent with some market specifications.  
The no-arbitrage condition translates into the requirement that $(X, Y)$ form a martingale (under the pricing measure), i.e.~$\bE\sq{Y|X}=X$, or equivalently  
$$
\bE\sq{(Y-X)g(X)}=0 , \text{ for all } g\in C^0_b(B),
$$
where $C^0_b(B)$ denotes the set of continuous bounded functions from $B$ to $\bR$.
In this case, we need to ask that $B\coloneqq\cX=\cY$ has an additional vector space structure, and thus we assume that $B$ is a separable Banach space. Moreover, since $X, Y$ need to be integrable, we must ask that their laws $\mu, \nu$ have finite first moments, i.e.~$\mu, \nu\in \cP_1(B)$. 
As usual,  $L^1\pp{\bP, B}$ denotes  the set of of integrable $B$-valued random variables.
Recall that the conditional expectation of a $B$-valued random variable is  defined and studied in \cite[Section 5.3]{Str10}, making use of Bochner's theory of integration. 
 
In MOT one considers as the set $\Pi_{\text{c}}\pp{\mu, \nu}$ of constrained transports the set $\cM\pp{\mu, \nu}$ of \emph{martingale couplings} (a.k.a.~\emph{martingale transports}) between $\mu$ and $\nu$, i.e.~laws of $(X, Y)$ such that $(X, Y)$ is a martingale, $X\sim \mu, Y\sim \nu$.   In more analytic terms,
$$ 
\cM\pp{\mu, \nu} = 
\gp{ \pi \in \Pi\pp{\mu, \nu}:  \int g(x)(y-x) d \pi(x, y)=0 \quad \text{ for all } g\in C^0_b(B)} .
$$
The existence of a martingale coupling is not always guaranteed as in classical OT; however, several equivalent characterisations of viability are known; to introduce the most important one, we need a definition.
\begin{definition}
\label{Definition: increasing convex order}
We say that probabilities $\mu, \nu$ on $B$ with finite first moments  are \emph{increasing in convex order}  (denoted by $\mu  \leqcx \nu$) if
$$
\int_{B} f(x)d\mu(x) \leq \int_{B} f(x)d\nu(x) ,
$$
for holds every Lipschitz\footnote{One can equivalently ask such inequality to hold for all \emph{continuous} convex functions, because every such function $f$ is the increasing limit of Lipschitz convex functions $f_n$ ($f_n$ can be obtained by inf-convolution between $f$ and the map $x \mapsto n \|  x\|_B$).
However, we prefer using Lipschitz functions because in this case the  integrals $\int f d\mu, \int f d\nu$ cannot take the value $\infty$.} 
convex function $f:B \to \bR$.
\end{definition}

An application of Jensen's inequality implies that $\mu  \leqcx \nu$  is necessary to ensure that $(\mu, \nu)$ is a viable input for the MOT problem.  Moreover, sufficiency was proved by many authors in different settings; the case of $B$ separable Banach space is due to   Strassen \cite{Strassen1965}. To summarise, we have the following
\begin{theorem}[Strassen]
\label{th: Strassen}
Given probabilities $\mu, \nu \in \cP_1(B)$ on a separable Banach space $B$, $\mathcal{M}\pp{\mu, \nu}\neq \emptyset$ holds if and only if $\mu \leqcx \nu$.
\end{theorem}

\noindent Since the seminal papers \cite{Beiglbock2013}, a sizeable stream of papers has been devoted to studying MOT and its applications in mathematical finance \cite{Labordere2017} and references therein.

\section{Discrete methods for OT}
\label{Section: discrete methods}

\subsection{Discretisation of classic OT}
\label{Section: Classical OT}
\noindent Since the OT problem (\ref{Kantorovich problem}) is an infinite-dimensional LP (Linear Program), a tempting idea for  solving it numerically is  to reduce to the case in which $\mu, \nu$ are finitely-supported, so as to make (\ref{Kantorovich problem}) a finite-dimensional LP, which  can then be easily solved numerically with several algorithms. With this in mind, one wants to construct a sequence of finitely-supported measures $\pp{\mu_n}_{n\in\bN},\pp{\nu_n}_{n\in\bN}$  converging to the target measures $\mu, \nu$, in order to calculate  $P\pp{\mu, \nu}$ as the limit of $P\pp{\mu_n,\nu_n}$, since the latter can be calculated numerically.   This procedure  relies on the  fundamental stability result in OT  \cite[Theorem 5.20]{Villani2007}, which ensures the continuity of the map $(\mu, \nu)\longmapsto P(\mu, \nu)$ with respect to the Wasserstein distance. 

Let's explicitly formulate the discrete OT problem (i.e.~the OT problem relative to finitely-supported measures $\mu, \nu$) as an LP. When $\mu, \nu$  are finitely-supported, they can be written as convex combinations of Dirac measures. If they are supported on at most $n\in \bN\setminus\gp{0}$ points,  we can write
\begin{align}
\label{convex if Dirac}
\begin{split}
\mu_n & = \sum_{i=1}^n \alpha_i^{}\delta_{x^{}_i}\;,\;\;\;\;\; \text{where} \quad  \alpha_i\geq 0 \quad  \forall i, \quad \sum_{i=1}^n\alpha^{}_i =1\;,  \\
\nu_n & = \sum_{j=1}^n \beta_j^{}\delta_{y^{}_j}\;,\;\;\;\;\; \text{where} \quad  \beta_j\geq 0 \quad  \forall j, \quad \sum_{j=1}^n\beta^{}_j =1 \;.
\end{split}
\end{align}
In this case, the cost function is defined as the matrix $c=\pp{c_{i, j}}_{i, j=1}^n$ where $c_{i, j} \coloneqq c(x_i,y_j) \in \bR_+$, and the feasible set is the polytope
\begin{flalign*}
	\Pi\pp{\mu_n,\nu_n} &\coloneqq
	\left\lbrace p \in \bR_+^{n\times n}: \; \sum_{j=1}^n p_{i, j} = \alpha_i \;; \;\;
	\sum_{i=1}^n p_{i, j} = \beta_j \;,\forall i, j=1,...,n  \right\rbrace .
\end{flalign*}
Notice, that we are denoting the discretised version of the transport (previous denoted by $\pi$) with the matrix $p \in \bR_+^{n\times n}$.

The discrete OT is then the following finite-dimensional LP 
\begin{equation}
\label{Discrete OT}
	\min \left\{\sum\limits_{i, j=1}^n p_{i, j}c_{i, j} :\ 
	p \in \bR_+^{n\times n},\  
	p \mathbbm{1}_n = \alpha,\
	p^T \mathbbm{1}_n = \beta \right\},
\end{equation}
\noindent where $\mathbbm{1}_n \in \bR^{n\times 1}$ denotes the column  vector with all entries equal to one, $\alpha=\pp{\alpha_1,...,\alpha_n}^T$ and $\beta=\pp{\beta_1,...,\beta_n}^T$. At this stage, computational methods for LPs can be applied to solve (\ref{Discrete OT}), see \cite{PeyreGabriel2019}.
Among these methods, we quote the well known network simplex \cite{Ahuja1993}, hungarian \cite{Kuhn} and the auction \cite{Bertsekas1981} algorithms. As the  above methods have complexity (at least) $O(n^3)$, this leaves the door open to faster and more efficient numerical methods.  
For example, a well-known technique to tackle with optimisation problems, it is  to add a regularisation term to the original objective function in order to obtain an approximating version of the problem which has an improved general structure and it is relatively easier to solve. In the setting of OT, \cite{Benamou2015} (see also \cite{Cuturi2013} and \cite{PeyreGabriel2019}) proposed an entropic regularisation which they showed to be solved by a fast and simple numerical scheme, namely the iterative Bregman algorithm \cite{Bregman1967}.
The very same regularisation can also be applied to the MOT problem although the algorithm can not be obtained in closed form (due to the additional martingale constraint). Indeed, \cite{DeMarch2018} proposed an extension of the iterative Bregman algorithm based on the \emph{dual} formulation of the problem in a general multidimensional setting.
We also mention that there are several other unrelated numerical methods to solve (\ref{Kantorovich problem}) \cite{MerigotThibert2021} which do not rely on the discretisation of $\mu, \nu$.

\subsection{Discretisation of constrained OT}
\label{Section: cOT}
Let's now move on to the constrained OT problem framework. Consider the problem (\ref{eq: cOT}), where $\Pi_{\text{c}}\pp{\mu, \nu}$ is the set of  transports $\pi\in \Pi_{\text{c}}\pp{\mu, \nu}$ from $\mu$ to $\nu$ which additionally satisfy and some linear (or convex) constraints. Mimicking the discrete method from the classical OT framework, we are interested in approximating $P_{\text{c}}\pp{\mu, \nu}$ by $P_{\text{c}}(\mu_n,\nu_n)$, where $\pp{\mu_n}_{n\in\bN}$, $\pp{\nu_n}_{n\in\bN}$ are sequences of finitely-supported measures converging to $\mu$, $\nu$ respectively. 

We are interested in whether it is possible to somehow adapt the discretisation method used in OT to the more general setting of cOT, and in particular in the following questions:
\begin{enumerate}[label=\textbf{Q.\arabic*}]
\item
If $(\mu, \nu)$ is viable, does there exist a sequence $\pp{\mu_n, \nu_n}_n$ of  finitely-supported probabilities such that $\pp{\mu_n, \nu_n}_n$ converges to $(\mu, \nu)$ (in some appropriate sense, e.g., in the weak topology)?
	\label{Q1}
\item
Can $\pp{\mu_n, \nu_n}$ be explicitly computed? How?
\label{Q2}
\item
Given $\pp{\mu_n, \nu_n} \to (\mu, \nu)$ as in \ref{Q1}, does $P_{\text{c}}\pp{\mu_n, \nu_n} \to  P_{\text{c}}\pp{\mu, \nu}$? 
\label{Q3}
\end{enumerate} 

In the following sections, we provide some possible answers to the questions above, in the setting of the MOT problem in arbitrary dimension. We point out that in this case the answers are quite delicate, since the map $\pp{\mu, \nu}\longmapsto  P_{\text{c}}(\mu, \nu)$ is \emph{not} continuous (unless $B=\bR$) \cite{BrJu22}, and it can happen that $\Pi_{\text{c}}\pp{\mu, \nu} \neq \emptyset$ whereas $\Pi_{\text{c}}\pp{\mu_n, \nu_n} = \emptyset$ for some $n$, for some sequence $\pp{\mu_n, \nu_n} $ converging to $\pp{\mu, \nu}$ \cite[Theorem 2.1]{Baker2015}.

\section{Discretisations preserving the convex order}
\label{Section: discretisations preserving the convex order}

In the setting of MOT, Theorem \ref{th: Strassen} shows that question \ref{Q1} becomes:\ 
\textit{is it possible to approximate $\pp{\mu, \nu}$ with finitely-supported $\pp{\mu_n, \nu_n}$ while preserving the convex order?}
This question is highly non-trivial, and has already been considered by several authors, as discussed in the introduction. 
In the next subsection, we discuss in some detail Baker's approach, so that we are later able to show in Remark \ref{Remark: U vs martingale quantisation} that it is a special case of our approach; we also introduce and develop our new approach, providing an explicit construction, bounds on the speed of convergence, and even some (rather weak) stability result.

\subsection{The $\cU$\emph{-quantisation}}
\label{Section: Known discretisations}
In \cite{Baker2015}, Baker considered the case $B=\bR$, and showed that in general $\cM(\mu_n,\nu_n)= \emptyset$ when the approximating measure $\mu_n, \nu_n$ are defined as the $L^2$-quantiser of $\mu$ and $\nu$, even if $\cM(\mu, \nu) \neq \emptyset$. He then proposed an approximation, called $\cU$\emph{-quantisation}, which preserves the convex order. 
 Given $\mu, \nu \in \mathcal{P}_2\pp{\bR}$ such that $\mu \leqcx \nu$, let $F_\mu(x) = \mu\big((-\infty,x]\big)$ be the (cumulative) distribution function of $\mu$, and recall that the quantile  function of $\mu$ is the generalised inverse 
\begin{align*}
	F_\mu^{-1} &:[0,1] \longrightarrow [-\infty,\infty]\\
	&:p \longmapsto \inf \big\lbrace x \in \bR: p \leq F_\mu(x) \big\rbrace .
\end{align*}
Then, for $n\in\bN \setminus \{0\}$, the $\cU$-quantisation of $\mu$ of order $n$ is defined by
\begin{equation}
\label{U quantisation}
	\hat{\mu}_n  \coloneqq 
	\cU_n(\mu)  \coloneqq \frac{1}{n} \sum_{i=i}^n \delta_{x_i},
	\quad \text{where} \quad
	x_i\coloneqq
	n\int_{\frac{i-1}{n}}^{\frac{i}{n}}F^{-1}_\mu(u)du ,
\end{equation}
and it is a probability supported on $n$ points.
The $\cU$-quantisation operators $\cU_n:\cP_2\to \cP_2$ preserve the convex order and converge pointwise to the identity, meaning that 
\begin{enumerate}
	\item  $\mu \leqcx \nu$ implies $\cU_n(\mu) \leqcx \cU_n(\nu)$ for all $n$ \cite[Theorem 3.5]{Baker2015}.
	\item  $\cU_n(\mu) $   converge to  $\mu$ in\footnote{To be precise, \cite{Baker2015} only proves weak convergence; the stronger convergence in $\cW_1$ follows from our Remark \ref{Remark: U vs martingale quantisation} and Theorem \ref{th: main result on pms}.} 
	$\cW_1$  as $n\to \infty$  for all $\mu\in \cP_1(\bR)$ \cite[Theorem 3.6]{Baker2015}.
\end{enumerate} 
We point out the curious formal analogy between $\cU_n(\mu)$ in (\ref{U quantisation}) and the empirical measure $\frac{1}{n}\sum_{i=1}^n \delta_{X_i}$ considered in  \cite{Alfonsi2019IJTAF, Alfonsi2019}.

\subsection{The martingale quantisation}
\label{Section: The martingale quantisation}
Our proposed solution to \ref{Q1} for the MOT problem is summarised in the subsequent Theorem \ref{th: main result on pms}. We recall that $B$ is assumed to be a separable Banach space throughout the paper.

\begin{theorem}
\label{th: main result on pms} 
Let $p\in [1,\infty)$ and $\mu, \nu \in \cP_{p}\pp{B}$ such that $\mu \leqcx \nu$. Then, there exist finitely-supported probability measures $\mu_n, \nu_n\in \cP_p(B), n\in\bN$,  such that $\mu_n \leqcx \nu_n$ for all $n \in \bN$ and $\mu_n \to \mu, \nu_n \to \nu$ in $\cP_p$. Moreover, one can additionally obtain that $\mu_n \leqcx\mu$ and $\nu_n \leqcx \nu$ for all $n \in \bN$. 
\end{theorem}
We will prove Theorem \ref{th: main result on pms} as a simple corollary of the following analogous  result  involving random variables. We now introduce some notations and facts
used without further notice throughout the paper.\\

\noindent Given $\Pi, \Pi_1,\Pi_2$ finite partitions of a set $E$, and $X, Y:\Omega \to E$, we write
$$
\hspace{-0.3 cm} 
\#\Pi\coloneqq \text{card}(\Pi), 
\quad 
X^{-1}(\Pi)\coloneqq \{X^{-1}(P):  P \in \Pi\}, 
\quad  
\Pi_1\times \Pi_2 \coloneqq 
\gp{M\times N,\ M \in \Pi_1,\ N \in \Pi_2} ,
$$ 
so that $\#\Pi$ is the number of elements of $\Pi$ and $W^{-1}(\Pi)$ is a finite partition of $\Omega$. Clearly $\Pi_1 \times \Pi_2$ is a finite partition of $E \times E$ and 
$$
(X, Y)^{-1}\pp{\Pi_1 \times \Pi_2} = 
\gp{X^{-1}(M) \cap Y^{-1}(M) : M \in \Pi_1,  N \in \Pi_2} .
$$
In particular, taking $\Pi_2$ to be the trivial partition $\{E\}$ we get 
$$
X^{-1}\pp{\Pi_1} = (X, Y)^{-1}\pp{\Pi_1 \times \{E\}} .
$$
If $E$ is endowed with a $\sigma$-algebra $\cE$, and $n\in \bN\setminus\{0\}$, we will consider the family 
\begin{align*}
	\gp{\Pi=(P_i)_{i=1}^k \subseteq \cE: k\leq n, P_i\cap P_j=\emptyset \,\, \forall i\neq j, \cup_{i=1}^k P_i=E}, 
\end{align*}
of ($\cE$-measurable) partitions of $E$ which are made of at most $n$ elements.
Clearly, if $X$ is a $E$-random variable, then $X^{-1}(\Pi)$ is a partition of $\Omega$ made of at most $n$ elements. We recall that if $\cG\subseteq \cF$ is a $\sigma$-algebra on $\Omega$,  $G\in \cG$ is called an \emph{atom of $\cG$} if for every $H\in \cG$ either $G\cap H=G$ or $G\cap H=\emptyset$, and that $\cG$  is finite if and only if it is of the form $\cG=\sigma(H)$ for some finite partition $H\subseteq \cF$,  and in this case $H$ is the set of atoms of $\cG$, and $\cG$ is the family of all possible unions of sets in $H$ \cite[Proposition I.2.1]{Ne65Ma}.
 
\begin{theorem}
\label{th: main result on rvs} 
For $i=1,2,\ n\in \bN$, let $\Pi^n_i\subseteq \cB(B)$ be finite partitions  of $B$ which satisfy\footnote{Of course such partitions exist, since $B$ is separable. For example, choose  countably many Borel sets $(A_n)_{n\in \bN}$ which generate $\cB(B)$ (e.g.~the balls of radius $1/(n+1), n\in \bN$, centred at points in a countable dense set), and then take   $\Pi^{n}_1=\Pi^{n}_2$ as the family of the atoms of $\sigma\pp{\pp{A_i}_{i\leq n}}$.} 
\begin{align}
\label{eq: condition on pi^n_i}
	\sigma\bigg( \Pi^{n}_i, n \in \bN \bigg) = \cB\pp{B} , 
	\quad 
	\Pi^{n}_i \subseteq \Pi^{n+1}_i 	
\end{align} 
Let $p\in [1,\infty)$ and given $X,Y\in L^p(\bP,B)$ such that $\bE[Y|X]=X$, consider the `martingale quantisation' $\pp{X_n,Y_n}$ of $(X,Y)$, defined by 
\begin{align}
\label{eq: Xn, Yn}
	X_n \coloneqq 	
	\bE\sq{X \big\vert \sigma^n_X}, 
	\quad 
	Y_n \coloneqq 	
	\bE\sq{Y \big\vert \sigma^n_{X, Y}} , 
\end{align}
where
\begin{align}
\label{eq: sigma^n_X, sigma^n_X,Y}
	\sigma^n_X\coloneqq 
	\sigma(X^{-1}\pp{\Pi^n_1}) , 
	\quad 
	\sigma^n_{X, Y} \coloneqq 
	\sigma(X^{-1}\pp{\Pi^n_1})\vee \sigma(Y^{-1}\pp{\Pi^n_2})
\end{align}
i.e.~$\sigma^n_X,\sigma^n_{X,Y}$ are the $\sigma$-algebras generated by the partitions 
$$  
X^{-1}\pp{\Pi^n_1} =
(X, Y)^{-1}\pp{\Pi^n_1 \times \{B\}} , 
\quad 
(X, Y)^{-1}\pp{\Pi^n_1 \times \Pi^n_2} .
$$
Then $X_n$ (resp.~$Y_n$) is supported on at most $\#\Pi_1^n$ (resp.~$\#\Pi_1^n \cdot \#\Pi_2^n$) points, and
\begin{subequations}
\begin{align}
	\bE\sq{Y_n \vert \sigma^n_X} =\bE[Y_n|X_n]= 
	X_n , 
	\quad 
	X_n = E[X|X_n], 
	\quad 
	Y_n = E[Y|Y_n] ,
\\
\label{eq: conv Lp and ae}
 	X_n   \to X \text{ a.s.~and in } L^p, \quad 
	Y_n  \to Y \text{ a.s.~and in } L^p .
\end{align}
\end{subequations}
\end{theorem}
\begin{proof}
If a $\sigma$-algebra $\cG$ is generated by a partition made of $n$ sets and  $Z:\Omega \longrightarrow B$ is $\cG$-measurable, then $Z$ only takes at most $n$  values, since it is constant on every atom of $\cG$; thus, $X_n$ (resp.~$Y_n$) is supported on at most $\#\Pi_1^n$ (resp.~$\#\Pi_1^n \cdot \#\Pi_2^n$) points. 

The assumptions (\ref{eq: condition on pi^n_i}) imply that 
$(\sigma^n_X)_n,(\sigma^n_{X,Y})_n$ are filtrations on $\pp{\Omega, \cF}$ such that $\vee_n \sigma^n_X=\sigma(X),\ \vee_n \sigma^n_{X,Y}=\sigma(X,Y)$, and so 
\begin{align}
\label{eq: closure is limit}
\bE\sq{X\vert \vee_n \sigma^n_X  }=X, \quad \bE\sq{Y\big\vert \vee_n \sigma^n_{X, Y}}=\bE\sq{Y\vert X, Y}=Y.
\end{align}
By definition $\pp{X_n}_{n}$ (resp.~$\pp{Y_n}_{n}$) is a martingale with respect to the filtration $\pp{\sigma^n_X}_{n}$ (resp.~$(\sigma^n_{X,Y})_n$)  and is closed by $X$ (resp.~$Y$). 
Thus, (\ref{eq: closure is limit}) and the martingale convergence theorems \cite[Theorems 1.5, 1.14]{Pis16}  yield (\ref{eq: conv Lp and ae}).
The tower property  gives
\begin{equation}
	X_n =\bE\sq{\bE\sq{X|\sigma^n_X} \vert X_n} = 
	\bE\sq{X\vert X_n} , \quad 
	Y_n = 
	\bE\sq{\bE\sq{Y\vert\sigma^n_{X, Y}} \vert Y^n} = 
	\bE\sq{Y\vert Y_n} .
\end{equation}
Let us now prove $\bE\sq{Y_n \vert \sigma^n_X} =\bE[Y_n|X_n]=X_n$.
Let $\{A^n_i\}_{i}=X^{-1}(\Pi_1^n)$ be the family of atoms of $\sigma^n_{X}$. Then, $A^n_i\in \sigma^n_{X} \subseteq \sigma^n_{X, Y} \cap \sigma(X)$ and so
\begin{align*}
\textstyle 
	\bE\sq{Y_n \mathbbm{1}_{\{X\in A_i\}}} = 
	\bE\sq{Y \mathbbm{1}_{{\{X\in A_i\}}}} = 
	\bE\sq{X \mathbbm{1}_{{\{X\in A_i\}}}} = 
	\bE\sq{X_n \mathbbm{1}_{{\{X\in A_i\}}}} .
\end{align*} 
In other words, the measures $\bE\sq{Y_n \mathbbm{1}_{\gp{X\in \cdot }} }$ and $\bE\sq{X_n \mathbbm{1}_{\gp{X\in \cdot }} } $ coincide on the $\pi$-system\footnote{Trivially $\{\emptyset, A^n_i\}_{i}$ is closed under intersections since the atoms are all disjoint.} $\gp{\emptyset, A^n_i}_{i}$ which generates $\sigma^n_X$, and so they coincide on $\sigma^n_X$, which means  that $\bE\sq{Y_n \vert \sigma^n_X} = X_n$. It follows that $\sigma(X_n)\subseteq \sigma_X^n$, and so taking $\bE[\cdot |X_n]$, the tower property  gives
$$X_n=\bE[X_n|X_n]=\bE\sq{\bE\sq{Y_n\vert \sigma_X^n}\vert X_n}=\bE\sq{Y_n\vert X_n} .$$
\end{proof}

\begin{lemma}
\label{le: W_p bounded by L^p}
If $X, Y$  are $B$-valued random variables with laws $\mu, \nu\in \cP_p$ then 
$$
\cW_p\pp{\mu, \nu}  \leq \norm{X-Y}_{L^p\pp{\bP,B}} .
$$
\end{lemma}
\begin{proof}
If $(X, Y)$ have joint law $\pi$ then 
$$
\norm{X-Y}^p_{L^p\pp{\bP,B}}=\int \abs{x- y}^p d \pi(x, y) ,
$$
so the thesis follows trivially from the definition of $\cW_p(\mu, \nu)$.
\end{proof} 

\begin{proof}[Proof of Theorem \ref{th: main result on pms}]
By Theorem \ref{th: Strassen} there exist $X, Y\in L^p$ such that $\bE\sq{Y\vert X}=X$. Applying Theorem \ref{th: main result on rvs} to such $X, Y$ yields some $X_n, Y_n$, whose laws $\mu_n, \nu_n$ satisfy the desired properties by Lemma \ref{le: W_p bounded by L^p} and Theorem \ref{th: Strassen}. 
\end{proof}

\begin{remark}
\label{re: proofs are constructive}
Importantly, the proofs of Theorems \ref{th: main result on pms} and \ref{th: main result on rvs} provide an \emph{algorithm}  to construct the discrete random variables $X_n, Y_n$ which satisfy the martingale condition, and thus the discretisation $\pp{\mu_n,\nu_n}$ of $\pp{\mu, \nu}$ which preserves the convex order, thus providing an answer to both question \ref{Q1} and \ref{Q2}. 
\end{remark} 

\begin{remark}
\label{re: countable partitions}
The previous construction can be applied without any change to countable partitions of $B$, in which case one obtains sequences of random variables $X_n,Y_n$ and measures $\mu_n,\nu_n$, each of which is countably supported. 
\end{remark}

Given the above results, we can provide the following definition which summarises the proposed answer to \ref{Q1} in the setting of MOT.

\begin{definition}
\label{Definition: martingale quantisation on marginals}
We call the sequence of probability measures $\pp{\mu_n, \nu_n}_{n\in\bN}$ (or, equivalently, of random variables $\pp{X_n, Y_n}_{n\in\bN}$) a \emph{martingale quantisation} of the couple $\pp{\mu, \nu}$ (resp.~$\pp{X, Y}$).
\end{definition}

\begin{remark}
\label{Remark: Y_n depends on both X and Y}	
	It is important to underline that $X_n$ only depends on the target random variable $X$; on the other hand, the quantised random variable $Y_n$ depends both on $X$ and $Y$. This dependence seems to be necessary to ensure that $X_n$ and $Y_n$ satisfy the martingale condition in this general setting.
\end{remark}

\begin{remark}
\label{Remark: U vs martingale quantisation}
There is a close link between the $\cU$-quantisation, and the martingale quantisation. Indeed, given a probability $\mu$ on $\bR$, its quantile function $X\coloneqq F^{-1}_\mu$, seen as a measurable function on the probability space $\Omega\coloneqq(0,1)$ endowed with the Lebesgue measure $\cL^1=\bP$ on the Borel $\sigma$-algebra $\cB((0,1))$, is a random variable with law $\mu$. Thus, if one considers the partition $\Pi^n$ of $\bR$ given by the intervals  
$$
\Bigg(X\pp{0^+},X\pp{\frac{1}{n}}\Bigg], \
\Bigg(X\pp{\frac{i-1}{n}},X\pp{\frac{i}{n}}\Bigg]
\text{ for } i=1, \ldots, n-1 , \,  
\Bigg(X\pp{\frac{n-1}{n}},X(1^-)\Bigg) ,
$$
then the $\sigma$-algebra $X^{-1}\pp{\sigma\pp{\Pi^{n}_1}} $ appearing in (\ref{eq: sigma^n_X, sigma^n_X,Y}) is  the one generated by the partition 
\begin{align}
\label{eq: partition [0,1]}
	\Bigg(0,\frac{1}{n}\Bigg], 
	\Bigg(\frac{1}{n}, \frac{2}{n}\Bigg], \ldots,
	\Bigg(\frac{n-1}{n}, 1\Bigg)
\end{align}
of $\Omega=(0,1)$, and thus $X_n \coloneqq \bE\sq{X\vert\sigma^n_X} $ as defined in (\ref{eq: Xn, Yn}) takes the value $\int_{\frac{i-1}{n}}^{\frac{i}{n}}F^{-1}_\mu(u)du$ on the interval $\left(\frac{i-1}{n}, \frac{i}{n}\right] \cap (0,1)$, and so the law of $X_n$ is the probability $\cU_n(\mu)$ defined in (\ref{U quantisation}). Thus, the $\cU$-quantisation is a special case of the martingale quantisation, in which  one chooses $X$ as being the quantile function of $\mu$, and  $X^{-1}\pp{\sigma\pp{\Pi^{n}_1}}$ as being the finite $\sigma$-algebra whose atoms are listed in (\ref{eq: partition [0,1]}), so that in particular all the weights $\alpha_i$ in the representation $\mu = \sum_{i=1}^n \alpha_i\delta_{x_i}$ are equal to $\frac{1}{n}$. However, notice that the martingale quantisation $Y_n$ of $Y\sim \nu$ involves also $X$, and so its law is unrelated to $\cU_n(\nu)$; moreover, given $\mu \leqcx \nu$, while the quantile functions $X\coloneqq F^{-1}_\mu$ and $Y\coloneqq F^{-1}_\nu$ have laws $\mu$ and $\nu$, they do not\footnote{Indeed,  if $\mu$ has no atoms then $F_\mu$ has no jumps, so $F^{-1}_\mu$ is \emph{strictly} increasing, and so $\sigma(F^{-1}_\mu)$  equals the Borel $\sigma$-algebra $\cB((0,1))$. Thus $F^{-1}_\nu$ is  $\sigma(F^{-1}_\mu)$-measurable, and so $(F^{-1}_\mu, F^{-1}_\nu)$ is martingale  only if $F^{-1}_\mu=F^{-1}_\nu$, i.e.~only if $\mu=\nu$.} 
in general form a martingale (on $(0,1)$ endowed with $\cL^1$). 
\end{remark}

\subsection{Explicit representation of the martingale quantisation}
\label{Section: martingale quantisation}
In this section, we will provide an explicit integral expression for the martingale quantisation, i.e.~for the random variables $X_n, Y_n$ which appear in Theorem \ref{th: main result on rvs}, and their laws  $\mu_n, \nu_n$ which appear in Theorem \ref{th: main result on pms}, thus positively answering question \ref{Q2}. 

\begin{proposition}
\label{Proposition: explicit representation for martingale quantisation rvs} 
Let $X \in L^p(\bP,B),\ X\sim \mu$ and $\hat{X} \coloneqq \bE\sq{X \vert \hat{\sigma}_X}$ where $\Pi_1=(P_{1,i})_i$ is a finite partition  of $B$ and $\hat{\sigma}_X \coloneqq 
\sigma(X^{-1}\pp{\Pi_1})$.
Then,
\begin{align}
\label{eq: quan bar expl repr X hat}
	\hat{X} &=  
	\sum_{i} x_i\mathbbm{1}_{P_{1,i}} ,
\end{align}
so that its law is given by 
\begin{align}
\label{eq: meas quan bar expl repr mu hat}
	\hat{\mu} &= 
	\cL\pp{\hat{X}} = 
	\sum_{i} \mu\pp{P_{1,i}} \delta_{x_i} ,
\end{align}
where the sum is over $i$ such that $\mu\pp{P_{1,i}}>0$ and $x_i$ can be expressed as 
\label{eq: barycenters}
\begin{align}
\label{eq: barycenters xi}
	x_i & \coloneqq 
	\frac{1}{\pi\pp{P_{1,i}\times B}}\int_{P_{1,i} \times B}x\,\pi\pp{dx, dy} = \frac{1}{\mu\pp{P_{1,i}}}\int_{P_{1,i}}x\,\mu\pp{dx} ,
\end{align} 
so that $x_i$ is the barycentre of the probability $\mu_i \coloneqq \frac{\mathbbm{1}_{P_{1,i}}}{\mu\pp{P_{1,i}}}\cdot \mu$.

\noindent Assume moreover that $Y \in L^p(\bP,B),\ Y\sim \nu$ satisfies $\bE[Y|X]=X$, and define $\hat{Y}\coloneqq \bE\sq{Y \vert \hat{\sigma}_{X,Y}}$ where  $\Pi_1=(P_{1,i})_i,\  \Pi_2=(P_{2,j})_j$ are finite partitions  of $B$ and $$\hat{\sigma}_{X, Y} \coloneqq \sigma\pp{X^{-1}\pp{\Pi_1}} \vee \sigma\pp{Y^{-1}\pp{\Pi_2}}.$$
Then,
\begin{align}
\label{eq: quan bar expl repr Y hat}
	\hat{Y} &= 
	\sum_{i, j} y_{i, j}\mathbbm{1}_{P_{1,i} \times P_{2,j}} ,
\end{align}
so that its law is given by 
\begin{align}
\label{eq: meas quan bar expl repr nu hat}
	\hat{\nu}  &=  
	\cL\pp{\hat{Y}} = 
	\sum_{i, j} \pi\pp{P_{1,i} \times P_{2,j}} \delta_{y_{i, j}} ,
\end{align}
where $\pi =\cL\pp{X, Y}$ and the sum is over $i,j$ such that $\pi\pp{P_{1,i}\times P_{2,j}}>0$. In addition, $y_{i,j}$ can be expressed as
\begin{align}
\label{eq: barycenters yij}
	y_{i, j} &  \coloneqq 
	\frac{1}{\pi\pp{P_{1,i}\times P_{2,j}}}\int_{P_{1,i} \times P_{2,j} }y\,\pi\pp{dx, dy} =
	\int_{B} y\,\nu_{i, j}\pp{dy} ,
\end{align} 
so that $y_{i,j}$ is the barycentre of the law $\nu_{i,j} \coloneqq \pi_{i,j} \circ P^{-1}_y$ of $P_y$ under $\pi_{i,j}$, where
\begin{equation}
	\pi_{i,j} \coloneqq \frac{\mathbbm{1}_{P_{1,i}\times P_{2,j}}}{\pi\pp{P_{1,i}\times P_{2,j}}} \cdot \pi
\end{equation}
In particular, the  $X_n, Y_n$ which appear in Theorem \ref{th: main result on rvs}  admit the  explicit representation given by eqs. (\ref{eq: quan bar expl repr X hat}) (\ref{eq: barycenters xi}) (\ref{eq: quan bar expl repr Y hat}) and (\ref{eq: barycenters yij}), in which the quantities $P_{1,i},  P_{2,j}, x_i, y_{i,j}$ also depend on the index $n$.
\end{proposition}

\begin{remark}
Notice that the above expressions  for $\hat{\mu}, \hat{\nu}$ do not depend on the random variables $X, Y$, they only depend on their joint law $\pi\in\cM\pp{\mu, \nu}$. However, such law is not an input of the MOT problem (only $\mu, \nu,c$ are), so it not always known in advance; when no such martingale coupling between $\mu$ and $\nu$ is known, one cannot apply the previous formulas to explicitly compute the laws $\mu_n,\nu_n$. As discussed in the introduction, this is a disadvantage of our algorithm, since currently algorithms which output a $\pi\in\cM\pp{\mu, \nu}$ given $\mu \leqcx \nu$ are only known in dimension 1.
\end{remark}

To prove the above proposition, we will need the following simple lemma. 
 
\begin{lemma}
\label{le: law of Z}
If $(E,\cE)$ is a measurable space,  $\Pi=\{P_j\}_j \subseteq \cE$ a finite partition of $E$, $W$ a $E$-valued random variable, $G\in L^1(\bP,B)$, and $\theta=\cL(W,G)$ then 
$$\bE[G|\sigma(W^{-1}(\Pi))]=\sum_{j: \bP(W\in P_i)>0} 1_{P_j}(W)\frac{\bE[1_{P_j}(W)G]}{\bP(W\in P_i)}$$
has   law $\zeta$  given by the formula
$$ 	\zeta  = \sum_{j} \theta(P_j\times B) \delta_{z_j}, \quad \text{ with } z_j\coloneqq\frac{1}{\theta(P_j\times B)}\int_{P_j \times B} g \,d\theta(e,g) .
$$
 \end{lemma}
\begin{proof}
The formula for $\bE[G|\sigma(W^{-1}(\Pi))]$ is simply the definition of conditional expectation with respect to the $\sigma$-algebra generated by the partition $W^{-1}(\Pi)$. 
That $\zeta$ are as stated follows from
 the fact that the random variable $\sum_i 1_{H_i} w_i$ has law $\sum_i \bP(H_i)\delta_{w_i}$.
\end{proof} 

\begin{proof}[Proof of Proposition \ref{Proposition: explicit representation for martingale quantisation rvs}]
Applying Lemma \ref{le: law of Z} with  $G=W\sim \mu$, we get that $d\theta(e,g)=d\mu(e) d \delta_{e}(g)$, and so $\hat{\mu}$ is given by eqs. (\ref{eq: meas quan bar expl repr mu hat}) and (\ref{eq: barycenters xi}).
Applying instead Lemma \ref{le: law of Z} with
$$E\coloneqq  B \times B, \quad \Pi\coloneqq \Pi_1 \times  \Pi_2, \quad  W=(X,Y)\sim \pi ,	
$$ 
writing $e\in E$ as $e=(x,y)$ and using $\theta(\cdot \times B) =\pi$ gives
$$ 
\hat{\nu}  = \sum_{i,j} \pi(P_{i,i}\times P_{2,j}) \delta_{y_{i,j}}, \quad \text{ with } y_{i,j}\coloneqq\frac{1}{\pi(P_{i,i}\times P_{2,j}) }\int_{P_{i,i}\times P_{2,j} \times B} g \,d\theta(x,y,g) .
$$
If then one  takes $G=Y$, it follows that $d\theta(x,y,g)=d\pi(x,y) d \delta_{y}(g)$, and so we get that $\hat{\nu}$ is given by eqs. (\ref{eq: meas quan bar expl repr nu hat}) and (\ref{eq: barycenters yij}).
\end{proof}

\subsection{Bounds on the speed of convergence}
We will now show that the average diameter of the the elements of $\Pi_1$ (resp.~$\Pi_2$) with respect to $\mu$ (resp.~$\nu$)  is an upper-bound for the Wasserstein-1 distance $\cW_1(\mu,\hat{\mu})$ (resp.~$\cW_1(\nu,\hat{\nu})$) between the laws of $X$ and $\hat{X}$ (resp.~$Y$ and $\hat{Y}$) which appear in Proposition \ref{Proposition: explicit representation for martingale quantisation rvs}. We remind the reader that  the \emph{diameter} of $E\subseteq B$ is defined as 
$$
\text{diam}\pp{E} \coloneqq \sup_{x, y\in E}\norm{x-y}\in [0,\infty] .
$$

\begin{theorem}
\label{th: speed of convergence}
Under the assumptions of Proposition \ref{Proposition: explicit representation for martingale quantisation rvs}, the following estimates hold:
\begin{subequations}
\label{eq: bound W1 barycentric quantisation both}	
\begin{align}
\label{eq: bound W1 proper barycentric quantisation}
	\cW_p\pp{\mu,\hat{\mu}} \leq \norm{X - \hat{X}}_{L^p\pp{\bP,B}} \leq 
	\sum_i (\text{diam}\pp{P_{1,i}})^p \mu\pp{P_{1,i}} \\
	\label{eq: bound W1 barycentric quantisation}
	\cW_p\pp{\nu,\hat{\nu}} \leq \norm{Y - \hat{Y}}_{L^p\pp{\bP,B}} \leq   
	\sum_{j} (\text{diam}\pp{P_{2,j}})^p \nu\pp{P_{2,j}}) .
\end{align}
\end{subequations}
where the sum over $i$ (resp.~$j$) is over the $i$ such that  $ \mu\pp{P_{1,i}}>0$ (resp.~$j$ such that  $ \nu\pp{P_{2,j}}>0$). In particular, in the setting of Theorem \ref{th: main result on rvs}, calling $\mu,\nu,\mu_n,\nu_n$ the laws of $X, Y, X_n, Y_n$, we obtain the bounds
\begin{align}
\label{eq: bound W1 martingale quantisation}
	\cW_p\pp{\mu,\mu_n} \leq 
	\sum_i (\text{diam}\pp{P^n_{1,i}})^p \mu\pp{P^n_{1,i}} , \quad 		\cW_p\pp{\nu,\nu_n}   \leq   
	\sum_{j} (\text{diam}\pp{P^n_{2,j}})^p \nu\pp{P^n_{2,j}} .
\end{align}
\end{theorem}

To prove Theorem \ref{th: speed of convergence}, we need a lemma, which uses the well known fact that  the convex hull of $E$, denoted with $\text{co}(E)$,  can be obtained as $\text{co}(E)=\cup_{n\in \bN} \text{co}^n(E)$, where $\text{co}^n(E)$ is defined by induction as follows: 
$$\text{co}^1(E)\coloneqq \{tx+(1-t)y: x,y\in E, t\in [0,1]\}, \quad \text{co}^{n+1}(E)\coloneqq \text{co}^1(\text{co}^n(E)), n\in \bN .$$
We denote with $\overline{\text{co}}(E)$ the closed convex hull of $E$, which coincides with the closure $\overline{\text{co}(E)}$ of $\text{co}(E)$.
\begin{lemma}
\label{prop: diam closed convex hull}
If $F \subseteq E \subseteq B$ then $\text{diam}\pp{\overline{\text{co}}\pp{E}}=\text{diam}(E) \geq \text{diam}(F)$.
\end{lemma}
\begin{proof}
Trivially, $F \subseteq E \subseteq B$ implies
$\text{diam}(E) \geq \text{diam}(F)$. This implies that $(\text{co}^n(E))_n$ is increasing, and so  $\text{diam}(\text{co}(E))=\sup_{n} \text{diam}(\text{co}^n(E))$. Since the norm is continuous, $\text{diam}(E)= \text{diam}(\bar{E})$, so the rest of the thesis follows once we prove the inequality $ \text{diam}(E)\geq \text{diam}(\text{co}^1(E))$. 
To prove it, given $\epsilon>0$ choose $z^1,z^2\in \text{co}^1(E)$ such that $\|z^1-z^2\|\geq \text{diam}(\text{co}^1(E)) - \epsilon$.
For $i=1,2$ write $z^i$ as
$$
z^i=t^i x^i+(1-t^i)y^i , 
\quad 
\text{ where }
x^i,y^i\in E, t^i\in [0,1] .
$$
Choose $a, b \in F\coloneqq \{x^1,x^2,y^1,y^2\}$ such that
$$
\norm{a-b} = \text{diam}(F)\eqqcolon d\leq \text{diam}(E) .
$$ 
If $c\coloneqq \frac{a+b}{2}$ is the mid-point between $a$ and $b$, and $r\coloneqq d/2$, then $F$ is a subset of $\overline{B}_{r}(c)\coloneqq \gp{z: \abs{z-c}\leq r}$. Since $\overline{B}_{r}(c)$ is convex and $F\subseteq \overline{B}_{r}(c)$,  we have $z^1, z^2\in \overline{B}_{r}(c)$, and so $\norm{z^1-z^2}\leq d$.  We have proved that
$$ \text{diam}(E)\geq d \geq \norm{z^1-z^2} \geq  \text{diam}(\text{co}^1(E)) - \epsilon, \text{ for all } \epsilon>0,$$
and taking $\epsilon\downarrow 0$, we conclude.
\end{proof} 

In the course of the following proof, we will repeatedly make use of the fact that, if $C\subseteq B$ is closed and $\mu$ is a probability on $C$ then its barycentre $\text{bar}(\mu)\coloneqq \int x \mu(dx)$ belongs to $C$; this well known fact easily follows from Hanh-Banach's theorem, as in~\cite[Proposition 2.39]{LuJaMaNeSp10}. 

\begin{proof}[Proof of Theorem \ref{th: speed of convergence}]
 If $P_y$ denotes the projection of $B \times B$ onto its second coordinate, the explicit expression for $\hat{Y}$ gives that 
\begin{align}
\label{eq: value gen bar}
&	\norm{Y -\hat{Y}}_{L^p\pp{\bP,B}} = 
	\sum_{i, j} \int_{P_{1,i}\times P_{2,j}} 
	\norm{y- P_y(\text{bar}\,\pp{\pi_{i,j}}) }^p \pi(dx,dy) .
\end{align} 
Since  $\pi_{i,j}\pp{\pp{P_{1,i}\times P_{2,j}}^c}=0$, we have that $\text{bar}\pp{\pi_{i,j}}\in \overline{\text{co}}\pp{P_{1,i}\times P_{2,j}} $, and since
$$ \overline{\text{co}}(E \times F)=\overline{\text{co}}(E ) \times \overline{\text{co}}(F), \quad \text{ for all } E,F \subseteq B,$$
we get that $P_y(\text{bar}\,\pp{\pi_{i,j}})\in \overline{\text{co}}(P_{2,j})$, and so 
$$ 	\norm{y- P_y(\text{bar}\,\pp{\pi_{i,j}}) } \leq \text{diam}(\overline{\text{co}}\pp{P_{2,j}}) , \quad (x,y) \in P_{1,i} \times P_{2,j}$$
from which, using Proposition \ref{prop: diam closed convex hull},  we get
$$ 	\norm{y- P_y(\text{bar}\,\pp{\pi_{i,j}}) } \mathbbm{1}_{P_{1,i} \times P_{2,j}}(x,y)\leq \text{diam}\pp{P_{2,j}} \mathbbm{1}_{P_{1,i} \times P_{2,j}}(x,y)$$
Taking $p$-powers, summing over $i,j$ and integrating with respect to $\pi$, using (\ref{eq: value gen bar}), we finally get (\ref{eq: bound W1 barycentric quantisation}).
Finally, (\ref{eq: bound W1 proper barycentric quantisation}) follows from applying 
 (\ref{eq: bound W1 barycentric quantisation}) not to $X,Y,\Pi_1,\Pi_2$, but rather to 
 $$X'\coloneqq \bE[X], \quad Y'\coloneqq X, \quad \Pi_1'\coloneqq \{B\}, \quad \Pi_2'\coloneqq \Pi_1 ,$$
which we can do since\footnote{If this reasoning seems puzzling, it should become clear after reading section \ref{Section: Barycentric quantisation}, in particular Theorem \ref{thm: estimates speed quant bar}.} $\bE[Y'|X']=X'$.
\end{proof}

Notice that, although the bounds proved in Theorem \ref{th: speed of convergence} have the pleasing quality of being easy to compute numerically for partitions made `simple' sets, they are not sharp: it is possible to find $\Pi_1^n, \Pi_2^n$ such that the upper-bounds do not converge to 0 even if $(\mu_n)_n, (\nu_n)_n$ converge to $\mu, \nu$, and the upper-bounds can even be vacuous (i.e.~they equal $+\infty$) for some $\pi,\Pi_1^n, \Pi_2^n$. 

\begin{remark}
\label{re: countable partitions}
We could have worked identically if considering partitions which are countable, instead of finite. This can be important, because $B$ admits a countable   partitions made of bounded sets, in fact for each $\epsilon>0$ there exists a countable partition made of sets of diameter at most $\epsilon$.
Using such partitions is useful since one can also replace the bounds (\ref{eq: bound W1 barycentric quantisation both}) with
\begin{align}
\label{ }
	\cW_p\pp{\mu,\hat{\mu}} \leq 
	\sup_i \; \pp{\text{diam}\pp{P_{1,i}}}^p , \quad 	
	\cW_p\pp{\nu,\hat{\nu}} \leq 
	\sup_j \; \pp{\text{diam}\pp{P_{2,j}}}^p ,
\end{align} 
which are less precise, but have the interesting feature that the upper-bounds do not depend on $\mu,\nu$. Of course,  one could achieve the same result working with finite partitions whose elements which intersect the  supports of the measures are bounded, and such partitions always exist when the measures are compactly supported. In particular, if $\mu,\nu$ are compactly supported probabilities on $B=\bR^N$, and the partitions are made of sets whose diameter converges to zero (e.g.~hypercubes with sides of length $1/2^n$), the given bounds always go to 0.
\end{remark}

\subsection{Stability}
\label{Section: Stability}
Question \ref{Q3} addresses the so-called stability result for a general cOT problem. As already remarked, in the setting of MOT, this question is highly non trivial. 
Indeed, a positive answer in this direction has been provided on the real line setting from different perspectives \cite{Backhoff2019, Juillet2016, Wiesel2019} whereas only the very recent work of \cite{BrJu22} shows that the stability result does not hold for the MOT problem on $\bR^d, d\geq 2$.
Therefore, we reformulate question \ref{Q3} as follows:
\textit{if $\pp{\mu_n,\nu_n} \to \pp{\mu,\nu}$ 
under which additional assumptions do we have that}
	\begin{align}
	\label{eq: stability}
		P_n \coloneqq
		\inf_{\cM\pp{\mu_n,\nu_n}} \bE\sq{c\pp{X_n, Y_n}} \to \inf_{\cM\pp{\mu,\nu}} \bE\sq{c\pp{X, Y}} \eqqcolon 
		P\ ?
\end{align} 
Here, we provide a rather weak answer to this question leaving a more structured result for further investigation. 
In particular, this is linked with the fact that the martingale quantisation algorithm needs as an input an arbitrary (i.e.~not necessarily an optimal) $\tilde{\pi}=\cL(X, Y)\in\cM(\mu, \nu)$  and outputs $\pi_n = \cL\pp{X_n, Y_n}$.
One can choose such $\tilde{\pi}, \pi_n$ as shown in Lemma \ref{le: weak stability MOT}.\\

\noindent Recall that $f:B\times B \to \bR^+$ is said to have sub-linear growth if
\begin{equation}
\label{eq: sub-linear growth cost function}
	\abs{f\pp{x, y}} \leq C\pp{1+\norm{x}+\norm{y}}
\end{equation}
for some $C\in \bR$, i.e.~if $f \in C^0_p(B\times B)$ with $p=1$.

The following result is the martingale analogue of a well-known result in OT.

\begin{lemma}
\label{le: {pi_n} is tight and convergence}
Let $\tilde{\pi}\in\cM\pp{\mu,\nu},\ \mu_n \to \mu, \nu_n \to \nu,\ \pi_n \in \cM\pp{\mu_n,\nu_n}$ with $\pi_n \to \tilde{\pi}$. Then, the sequence $\pp{\pi_n}_n$ is tight so that it admits accumulation points.\ Any accumulation point belongs to $\cM\pp{\mu,\nu}$.
\end{lemma}
\begin{proof}
By Prokhorov's Theorem \cite[Theorem 5.2]{VanGaans2003}, $\pp{\mu_n}_n$ and $\pp{\nu_n}_n$ are (uniformly) tight, therefore, for every $\epsilon>0$ there exist compact sets $K^\mu_\epsilon,\ K^\nu_\epsilon \subseteq B$ such that $ \mu_n\pp{B\setminus K^\mu_\epsilon} \leq \epsilon/2,\ \nu_n\pp{B\setminus K^\nu_\epsilon} \leq \epsilon/2$. Then,
$$
\pi_n\pp{B\times B \setminus K^\mu_\epsilon \times K^\nu_\epsilon} \leq 
\mu_n\pp{B\setminus K^\mu_\epsilon} + 
\nu_n\pp{B\setminus K^\nu_\epsilon} \leq 
\epsilon ,
$$
so that $\pp{\pi_n}_n$ is tight and there exist $\underline{\pi} \in \cP\pp{B\times B}$ such that $\pi_n$ weakly converges to $\underline{\pi}$ (up to taking a subsequence without relabelling). In particular, for any $f\in C^0_b(B\times B)$, we have
\begin{align}
\label{eq: weak convergence in lemma}
	\int f(x,y)\pi_n(dx, dy) \to \int f(x,y) \underline{\pi}(dx, dy) .
\end{align}
Let $g ,h \in C^0_b(B)$. Choosing $f(x,y) = g(x)$ and $f(x,y) = h(y)$ in (\ref{eq: weak convergence in lemma}), we obtain that $\underline{\pi}\in\Pi\pp{\mu,\nu}$. Moreover, 
\begin{align*}
	\int \norm{(x,y)}\pi_n(dx, dy) & =
	\int \norm{x}\mu_n(dx) +
	\int \norm{y}\nu_n(dy) 
	& \to 
	\int \norm{x}\mu(dx) +
	\int \norm{y}\nu(dy) \\
 	& =
	\int \norm{(x,y)}\underline{\pi}(dx, dy) ,
\end{align*}
and so, by Remark \ref{Remark: equivalence convergence C0b C0p},  $\pi_n \to \underline{\pi}$.
Finally, since the map $(x,y) \mapsto g(x)(y-x)$  has sub-linear growth, we have that
\begin{align*}
	0 =
	\int (y-x)g(x)\pi_n(dx, dy) \to
	\int (y-x)g(x)\underline{\pi}(dx, dy) ,
\end{align*}
so that $\underline{\pi} \in \cM\pp{\mu,\nu}$.
\end{proof}

\begin{lemma}
\label{le: weak stability MOT}
Under the assumptions of Lemma \ref{le: {pi_n} is tight and convergence}, if $c:B\times B \to \bR^+$ is a continuous cost function with sub-linear growth, then
\begin{align}
\label{eq: weak stability MOT}
\begin{split}
	\bE^{\tilde{\pi}}\sq{c\pp{X, Y}} \geq 
	\limsup_n P_n &\geq 
	\liminf_n P_n \geq
	P .
\end{split}
\end{align}
In particular, (\ref{eq: stability}) holds along a minimising subsequence if 
\begin{equation}
\label{eq: condition stability}
\tilde{\pi} = \pi^*\in \text{argmin}_{\pi\in \cM\pp{\mu,\nu}} \bE\sq{c\pp{X, Y}} .	
\end{equation}
\end{lemma}

\begin{proof}
For $n\in \bN$, let $\pi_n$ be a martingale quantisation which can be obtained starting from the (sub-optimal) transport $\tilde{\pi}$.
Then, by the continuity of $c$ and assumption (\ref{eq: sub-linear growth cost function}), it follows that
$$
\bE^{\pi_n}\sq{c(X,Y)} \to 
\bE^{\tilde{\pi}}\sq{c(X,Y)} \geq \min_{\pi \in \cM\pp{\mu,\nu}} \bE^{\pi}\sq{c} = P .
$$
Let $\pi^*_n \in \cM\pp{\mu_n,\nu_n}$ be the minimiser. Then, by Lemma \ref{le: {pi_n} is tight and convergence}, $\pp{\pi^*_n}_n$ is tight and there exists $\underline{\pi}\in\cM\pp{\mu,\nu}$ such that $\pi^*_n \to \underline{\pi}$ (up to taking a subsequence without relabelling).
It then follows that
$$
\bE^{\pi_n}\sq{c(X,Y)} \geq 
\min_{\pi \in \cM\pp{\mu_n, \nu_n}}\bE^{\pi}\sq{c(X,Y)} =
P_n =
\bE^{\pi^*_n}\sq{c(X,Y)} \to 
\bE^{\underline{\pi}}\sq{c(X,Y)} ,
$$ 
Therefore,
$$
\bE^{\underline{\pi}}\sq{c(X,Y)} \geq 
\min_{\pi \in \cM\pp{\mu,\nu}}  \bE^{\pi}\sq{c(X,Y)} .
$$
Putting the pieces together, we get the claim.

If particular, if (\ref{eq: condition stability}) holds, i.e.~$\tilde{\pi} = \pi^*$ is the minimiser, then
\begin{align*}
	\limsup_n P_n &\leq
	\limsup_n \bE\sq{c\pp{X_n,Y_n}}\\  
	&=
	\bE^{\tilde{\pi}}\sq{c\pp{X,Y}}\\
	&\leq
	\inf_{\cM\pp{\mu,\nu}}\bE\sq{c\pp{X,Y}} = P ,
\end{align*}
where: the first inequality follows by definition; the equality follows from $c\pp{X_n, Y_n} \to c\pp{X, Y}$ $\bP$-a.s. and\footnote{Because $\gp{c\pp{X_n, Y_n}, n\in\bN}$ is uniformly integrable.} in $L^1\pp{B,\bP}$ since $c$ is assumed to be continuous and with sub-linear growth and finally, the last inequality follows from (\ref{eq: condition stability}).
\end{proof}

\begin{remark}
The existence of the optimiser $\pi^*$ is proved in \cite[Theorem 1.1]{Beiglbock2013} on the real-line and in \cite{Zaev2015} in the general setting of Banach spaces.
\end{remark}
\begin{remark}
Clearly, Lemma \ref{le: weak stability MOT} has limited usefulness, as it states that in order to approximate the optimal value one should already know the optimal $\pi^*$, in which case to calculate the optimal value it is simpler to compute $\bE^{\pi^*}\sq{c(X, Y)} $ rather than solving $\inf_{\cM\pp{\mu_n, \nu_n}} \bE\sq{c(X, Y)}$ and  computing its liminf. However, Lemma \ref{le: weak stability MOT} suggests a recursive scheme which might approximate the optimal value $P$ of the MOT problem and, at least, obtains smaller and thus more  precise optimal values.

Indeed, starting from the given transport  $\pi^{(0)}\coloneqq\tilde{\pi}\in\cM\pp{\mu,\nu}$, by Lemma \ref{le: weak stability MOT} we can construct $\pi^{(1)}\coloneqq \underline{\pi} \in \cM\pp{\mu, \nu}$ such that 
	$P^{(1)}\coloneqq 
	\bE^{\pi^{(1)}}\sq{c(X,Y)} \leq 
	\bE^{\pi^{(0)}}\sq{c(X,Y)} \eqqcolon 
	P^{(0)}$. Reiterating this step, one could obtain a sequence $\pp{P^{(l)}}_{l\in\bN}$ satisfying
$$
P^{(0)} \geq
\ldots \geq
P^{(l)} \geq
P^{(l+1)} \geq
\ldots \geq
P ,
\quad l \geq 1 .
$$	
\end{remark}

\section{Numerical examples}
\label{Section: Numerical examples}
In this section, we provide some examples in which we implement the proposed martingale quantisation scheme and solve the corresponding discretised MOT, illustrating cases in which (\ref{eq: stability}) seems to hold. 

\begin{example}
\label{Example: Example 1}
Consider an MOT problem on the real line between two uniform marginal distributions 
\begin{align*}
	\mu(dx) & = \frac{1}{2}\mathbbm{1}_{[-1,1]}(x) dx \\
	\nu(dy) & = \frac{1}{4}\mathbbm{1}_{[-2,2]}(y)dy ,
\end{align*}
and cost function given by
$c(x, y) \coloneqq \abs{y-x}^\rho,\ \rho = 2.3$, i.e.~
\begin{equation}
\label{eq: MOT Ex 1}		
	P \coloneqq \min_{\pi\in\cM\pp{\mu, \nu}}\int_{\bR\times\bR}\abs{y-x}^\rho \pi(dx,dy) .
\end{equation}
For this problem, it is known \cite[Example 6.1]{Alfonsi2019IJTAF} the expression of the martingale optimal transport 
\begin{equation}
\label{eq: martingale optimal transport Example 1}	
\pi^*(dx,dy) = \frac{1}{2}\mathbbm{1}_{[-1,1]}(x)\frac{\delta_{x+1}(dy) + \delta_{x-1}(dy)}{2}dx ,
\end{equation}
so that the corresponding optimal value of the problem is $P = \int_{\bR\times\bR} \abs{y-x}^\rho \pi^*(dx,dy) = 1$.\footnote{
The same holds as long as $\rho>2$. Analogously, one can consider the case $\rho<2$ and the MOT problem in (\ref{eq: MOT Ex 1}) with $\max$ in place of $\min$.} 
Firstly, we apply our martingale quantisation algorithm using as initial martingale transport $\tilde{\pi}$ the (sub-optimal) left-curtain coupling\footnote{
Identical results can be obtained using the right-curtain coupling
$$
{\small{
\pi^{rc}(dx,dy) \coloneqq 
\frac{1}{2}\mathbbm{1}_{[-1,1]}(x)\pp{\frac{3}{4} 
\delta_{\frac{3x}{2}-\frac{1}{2}}(dy) +\frac{1}{4}\delta_{\frac{3}{2} + \frac{x}{2}}}dx	.
}}
$$
} 
$$
\pi^{lc}(dx,dy) \coloneqq 
\frac{1}{2}\mathbbm{1}_{[-1,1]}(x)\pp{\frac{1}{4} 
\delta_{-\frac{x}{2}-\frac{3}{2}}(dy) +\frac{3}{4}\delta_{\frac{3x}{2} + \frac{1}{2}}(dy)}dx ,
$$
and two $n$-partitions  $\gp{P^n_{1,i}}_{i=1}^n, \gp{P^n_{2,j}}_{j=1}^n$ of supp$\pp{\mu}=\sq{-1,1}$ and supp$\pp{\nu}=\sq{-2,2}$, respectively. In particular, since for any integrable function $f$ and $P_i, P_j \subseteq \bR$ it holds
\begin{align*}
	\int_{P_i\times P_j} f(x, y) \tilde{\pi}(dx,dy) & = 
	\int_{P_i} \tfrac{1}{2}\gp{
    \tfrac{1}{4}f\pp{x,\pp{-\tfrac{x}{2}-\tfrac{3}{2}}}\mathbbm{1}_{P_j}\pp{-\tfrac{x}{2}-\tfrac{3}{2}}
    \color{white}}\\
    &\color{white}\gp{  
    \color{black}\quad\quad
	+\tfrac{3}{4}f\pp{x, \pp{\tfrac{3}{2}x + \tfrac{1}{2}}}\mathbbm{1}_{P_j}\pp{\tfrac{3}{2}x + \tfrac{1}{2}}}\color{black}dx ,
\end{align*}
we have that
\begin{align*}
	\omega^{n}_{X,i} 
	& = 
	\mu\pp{P^n_{1,i}} = \int_{P^n_{1,i}}\mu(dx) ,\\
	x^{n}_i 
	& = 
	\frac{1}{\omega^{n}_{X,i}} \int_{P^n_{1,i}} x \mu(dx) ,\\
	\omega^{n}_{Y,i, j} & = 
	\tilde{\pi}\pp{P^n_{1,i}\times P^n_{2,j}} 
	= \int_{P^n_{1,i}}  \tfrac{1}{2}\pp{\tfrac{1}{4}\mathbbm{1}_{P^n_{2,j}}\pp{-\tfrac{x}{2}-\tfrac{3}{2}}+ \tfrac{3}{4}\mathbbm{1}_{P^n_{2,j}}\pp{\tfrac{3x}{2}+\tfrac{1}{2}}}\mu(dx) ,\\
	y^{n}_{i, j} & = 
	\frac{1}{\omega^{(n_\mu)}_{Y, i, j}} 
	\int_{P^n_{1,i}\times P^n_{2,j}} y \tilde{\pi}(dx, dy) \\			
	& = \frac{1}{\omega^n_{Y, i, j}} 
		\int_{P^n_{1,i}}  \tfrac{1}{2}\pp{\tfrac{1}{4}\pp{-\tfrac{x}{2}-\tfrac{3}{2}}\mathbbm{1}_{P^n_{2,j}}\pp{-\tfrac{x}{2}-\tfrac{3}{2}}+ \tfrac{3}{4}\pp{\tfrac{3x}{2}+\tfrac{1}{2}}\mathbbm{1}_{P^n_{2,j}}\pp{\tfrac{3x}{2}+\tfrac{1}{2}}}\mu(dx) .
			\end{align*}
This allows us to generate finitely-supported $\pp{\mu_n}_n, \pp{\nu_n}_n$ and subsequently solve the primal LP corresponding to the discretised MOT $P\pp{\mu_n,\nu_n}$ as shown in Appendix \ref{Section: Appendix - algo primal LP}.
		
In Figure \ref{Figure: Discretisation of Example 1 - Values}, we plot the values $P_n\coloneqq P\pp{\mu_n,\nu_n}$ of the primal LP as a function of quantisation step $n$ which shows the numerical convergence of $P_n$ to $P=1$. 
The heat map of the optimiser is displayed in Figure \ref{Figure: Discretisation of Example 1 - heat map} for $n=100$. As expected, the optimiser $p_n$ (i.e.~the solution of the LP) is entirely concentrated on the lines $y=x\pm 1$ which represent the support of the martingale optimal transport $\pi^*$ in (\ref{eq: martingale optimal transport Example 1}).

\begin{figure}[!htbp]
\centering
	\begin{subfigure}[b]{.41\linewidth}
		\includegraphics[width=\linewidth, keepaspectratio]{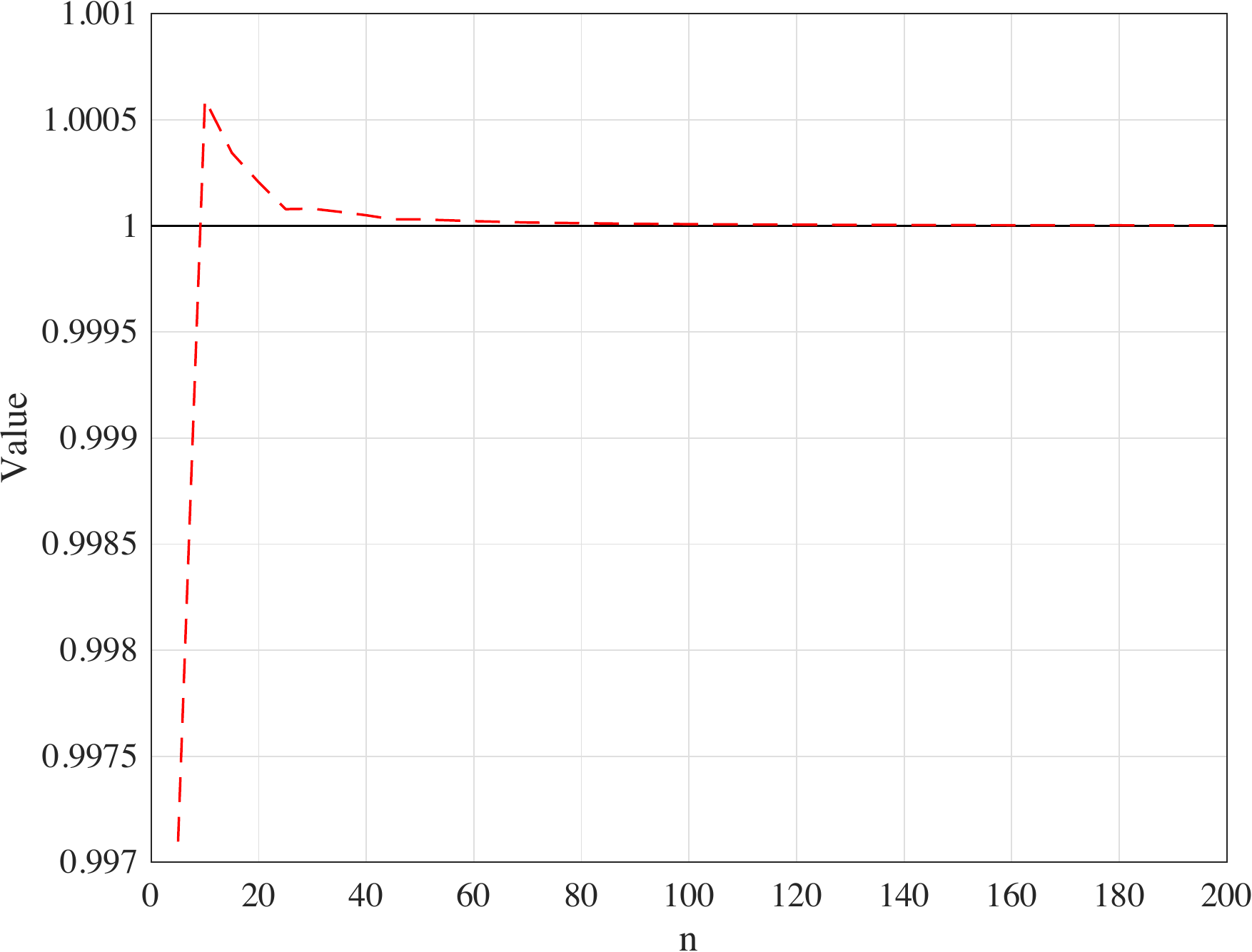}
    	\caption{}
    	\label{Figure: Discretisation of Example 1 - Values}
	\end{subfigure}
\hfil
	\begin{subfigure}[b]{.45\linewidth}
		\includegraphics[width=\linewidth, keepaspectratio]{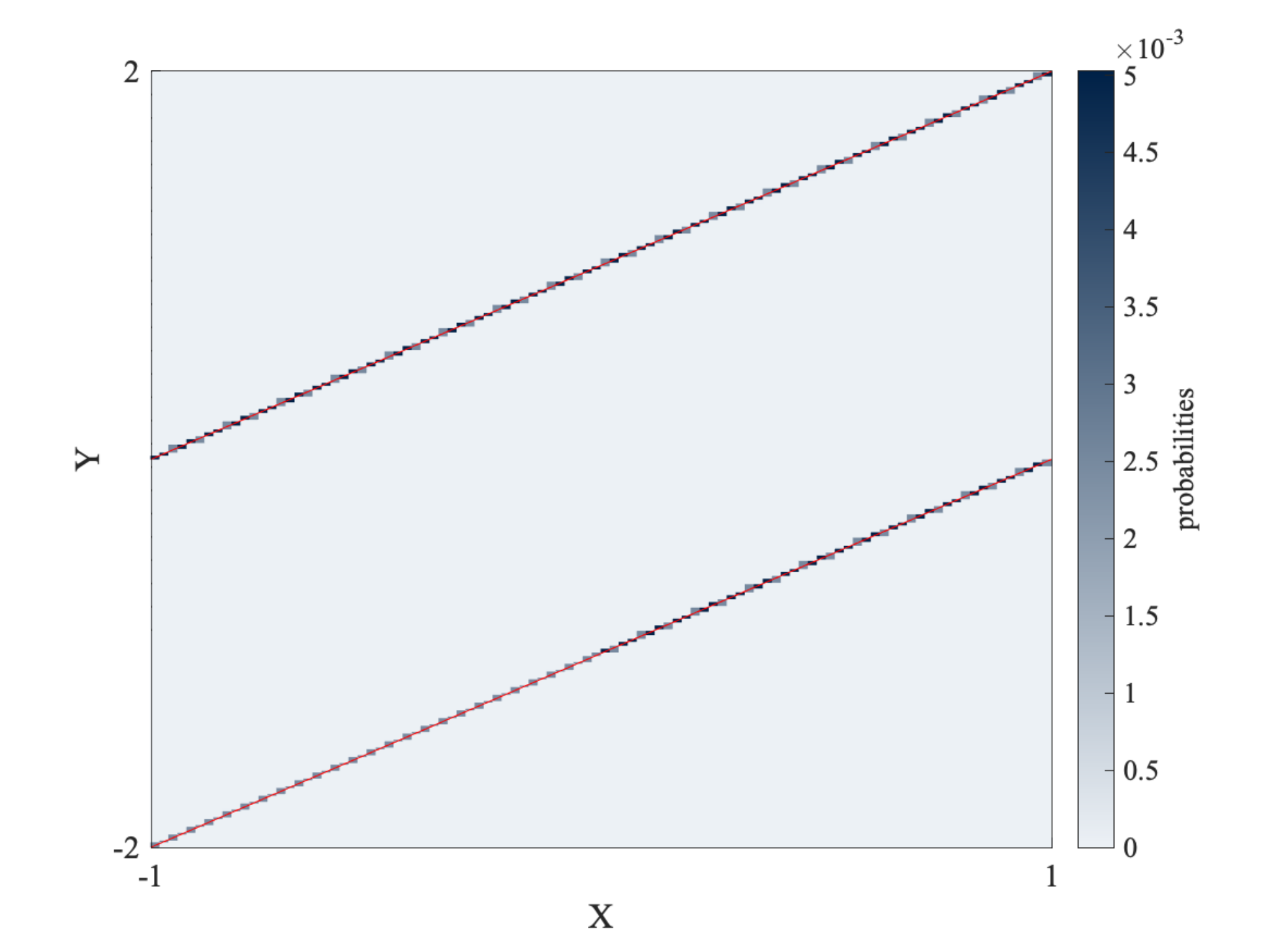}
    	\caption{}
    	\label{Figure: Discretisation of Example 1 - heat map}
	\end{subfigure}
\caption{Results of Example \ref{Example: Example 1}:
(a) Values of $P_n$ for $5\leq n \leq 200$;
(b) Heat map of the optimiser $p_n$ for $n=100$ (the two red lines $y=x\pm 1$ represent the support of the optimal transport $\pi^*$).}
\label{Figure: Discretisation of Example 1}
\end{figure}
\end{example}

\begin{example}
\label{Example: Gaussian 1d}
As a second example, we still consider an MOT problem on the real line with same cost function as in the previous example but with marginals given by two gaussian distributions
	$X\sim N(0,1)$ and $Y\sim N\pp{0,2}$. In this case, we do not know a priori a martingale transport $\tilde{\pi}$ between $\mu$ and $\nu$ in order to start our martingale quantisation scheme. However, we can easily obtain such one by exploiting the fact that if $Z\sim N(0,1)$ independent of $X$, we have that $Y= X + Z$ so that the law of $Y$ is given by the convolution between the law of $X$ and the law of $Z$. A martingale transport can be thus obtained as
	$$
	\tilde{\pi}(dx, dy) = \mu(dx)\eta(x+dy) \coloneqq \int_\bR\delta_{x+z}(dy)\eta(dz) ,
	$$
	where $\eta$ stands for the law of $Z$.
As before, we first apply our martingale quantisation scheme and then solve the corresponding LP.
We choose to quantise $\mu$ using the Voronoi ($L^2$-) quantisation so that $\gp{P^n_{1,i}}$ is given by optimal quantisation grid of the standard normal distribution with size $n$\footnote{Precomputed quantisation grids for $N\pp{0, I_d}$ for $d = 1, \dots, 10$ and $n = 1,\dots,104$ can be. downloaded from \href{www.quantize.maths-fi.com/}{www.quantize.maths-fi.com/}.}. Moreover, choosing $\gp{P^n_{2,j}} =\gp{P^n_{1,i}}$, the expression for $\nu_n$ can be easily obtained from
\begin{align*}
	\omega^n_{Y,i, j} & \coloneqq 
	\tilde{\pi}(P^n_{1,i}\times P^n_{2,j}) 
		 = \int_{P^n_{1,i}} \mu(dx)
		 \int_{P^n_{2,j}} \eta(x+dy) 
		 = \int_{P^n_{1,i}} \mu(dx)
		 \eta(x+P_{2,j}) \\
	y^{n}_{i, j} & = 
	\frac{1}{\omega^n_{Y,i, j}}\int_{P^n_{1,i}\times P^n_{2,j}}y\tilde{\pi}(dx,dy) 
	=\frac{1}{\omega^n_{Y,i, j}} \int_{P^n_{1,i}} \mu(dx)\int_{P^n_{2,j}}\eta(x+dy)y 
\end{align*}
Analogously as before, the results are presented in Figure \ref{Figure: Discretisation of Gaussian 1d}.

\begin{figure}[!htbp]
\centering
	\begin{subfigure}[b]{.41\linewidth}
	\includegraphics[width=\linewidth, keepaspectratio]{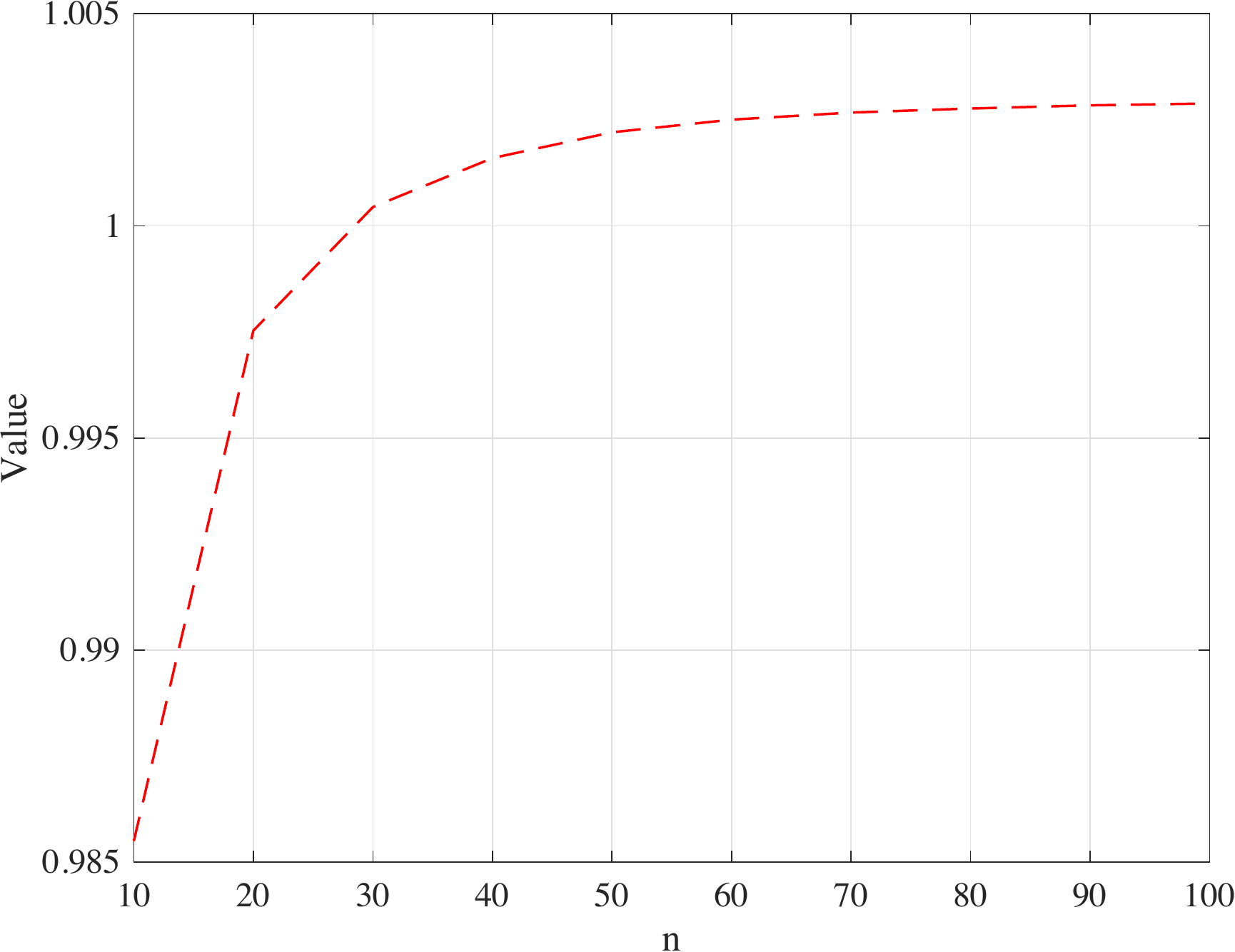}
 	\caption{}
 	\label{Figure: Discretisation of Gaussian 1d - Values}
	\end{subfigure}
	\hfil
	\begin{subfigure}[b]{.45\linewidth}
		\includegraphics[width=\linewidth, keepaspectratio]{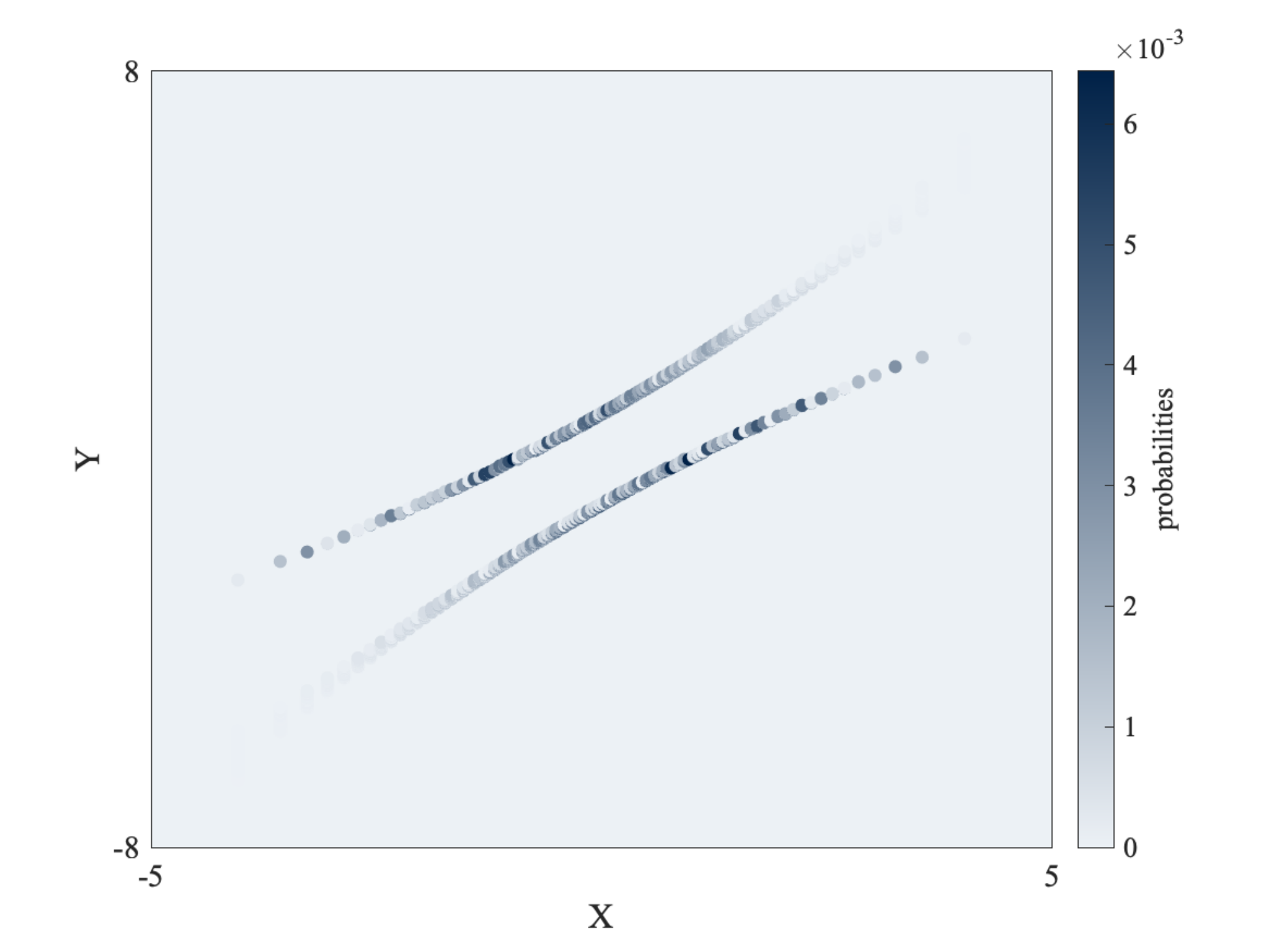}
 	\caption{}
 	\label{Figure: Discretisation of Gaussian 1d - Heat map}
 	\end{subfigure}
\caption{Results of Example \ref{Example: Gaussian 1d}:
(a) Values of $P_n $ for $10\leq n \leq 100$;
(b) Heat map of the optimiser $p_n$ for $n=100$.}
\label{Figure: Discretisation of Gaussian 1d}
\end{figure}
\end{example}

\begin{example}
\label{Example_3}
We now consider the equivalent of Example \ref{Example: Example 1} but with marginal distributions supported on $\bR^2$. This MOT problem was also studied in \cite[Example 5.2]{Alfonsi2019}. Consider $\mu, \nu$ be two uniform distribution on $\sq{-1,1}^2$ and $\sq{-2,2}^2$, respectively and the MOT problem with cost function $c(x, y) = 
	\abs{x_1-y_1}^\rho + \abs{x_2-y_2}^\rho$ where $x=\pp{x_1,x_2}, y=\pp{y_1,y_2}$ and $\rho = 2.3$. It can been shown that the martingale optimal transport is given by
	\begin{equation}
		\label{eq: martingale optimal transport Example 3}
	\pi^*\pp{dx,dy} = 
	\mu\pp{dx}\sum_{z}\delta_{\pp{x+z}}\pp{dz}
	\end{equation}
	where $z$ is a Rademacher distribution on $\bR^2$ i.e.~each $z_i$ is such that $z_i = 1$ or $z_i=-1$ with probability $1/2$.
Although it might seem trivial, we used the optimal transport $\pi^*$ as initial transport for the martingale quantisation and then solved the corresponding LP. In Figure \ref{Figure: Example_3}, we exhibit the heat-map of the points $\pp{x_2-x_1,\ y_2-y_1}$ under the optimiser $p_n$ for $n=100 $. The red lines ($y=\pm2, y=x$) represent the theoretical support of the projection of the optimal transport $\pi^*$ on $\pp{x_2-x_1,\ y_2-y_1}$.

\begin{figure}[!htbp]
\centering
	\includegraphics[width=.7\linewidth]{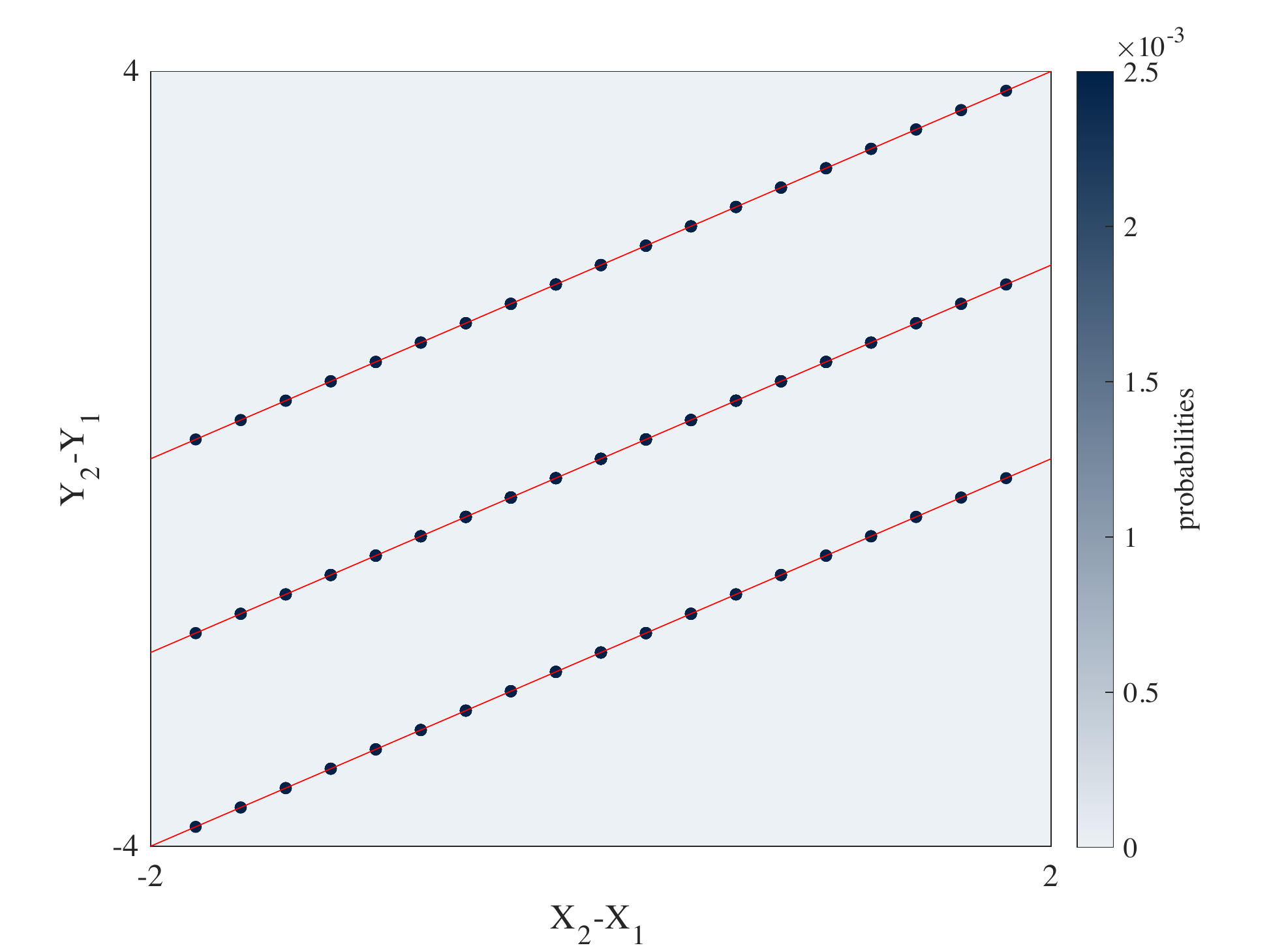}
\caption{Results of Example \ref{Example_3}: heat map of the projection of the optimiser $p_n$ on $\pp{x_2-x_1, y_2-y_1}$ for $n=100$ (the red lines $y=x\pm 2, y=x$ represent the theoretical support).}
\label{Figure: Example_3}
\end{figure}
\end{example}

\section{Barycentric quantisation}
\label{Section: Barycentric quantisation}
In the previous section, we saw how, given $\mu\leqcx \nu$ and partitions $\Pi_1,\Pi_2$ of $B$, we can explicitly build some $\hat{\mu},\hat{\nu}$ as laws of some $\hat{X},\hat{Y}$, which we build starting from $X\sim \mu,Y \sim \nu$ such that $\bE[Y|X]=X$. If for $i=1,2$ we choose $\Pi_i=\Pi_i^n$ for some filtration $(\Pi_i^n)_n$ which generates $\cB(B)$, the above construction provides the martingale quantisation. Since the constructions $(\mu,\nu)\mapsto (\hat{\mu},\hat{\nu})$ and $(X,Y)\mapsto (\hat{X},\hat{Y})$ have proved to be useful, in this section we define them in somewhat more abstract terms as `barycentric quantisations', studying their properties and showing in Proposition \ref{prop: alternative definition barycentric quantisation} and \ref{Proposition: barycentric quantisation span set fint supp measures smaller in convex order} that one can characterise which measures $\zeta$ such that $\zeta\leqcx \gamma$ are finitely-supported (resp.~supported on at most $n$ points) using barycentric quantisations of $\gamma$, but not using only proper barycentric quantisations (see Example \ref{Example: counterexample barycentric quantisation}).

\subsection{Preliminaries on quantisation theory}
Quantisation theory aims to represent a continuous signal with a discrete one trying to loose less information as possible. When the signal is described by a random variable $G$, and $n\in \bN$, a $n$-\emph{quantisation} of $G$ is a $\sigma(G)$-measurable random variable $Z$ which takes at most $n$ values. Equivalently $Z = q(G)$ for some Borel-measurable  function $q$ which takes at most $n$ values, i.e.~$q=\sum_{j=1}^n z_j 1_{P_j}$ for some $z_j\in B$ and a partition $P_1,\ldots, P_n$ of $B$ made of Borel sets.  We refer to \cite{Graf2000} for a comprehensive treatment of the quantisation theory in the finite-dimensional case (also known as vector quantisation theory) and \cite{PagesPrintems2009} for its extension to the infinite-dimensional setting.

If the signal is represented by a probability $\gamma$, we say that a probability $\zeta$ is a $n$-\emph{quantisation} of $\gamma$ if $\zeta=\sum_{j=1}^n \gamma(P_j) \delta_{z_j} $ for some $z_j\in B$ and a partition $P_1,\ldots, P_n$ of $B$ made of Borel sets. 
 
\begin{remark}
\label{re: quantisation of rv and measures}
 These two notions are of course closely linked:
\begin{enumerate}
\item  If $Z$ is a $n$-quantisation of $G$, the law  of $Z$ is a  quantisation of the law of $G$.
\item  If  $\zeta$ is a $n$-quantisation of $\gamma$, then there exist $G$ with law $\gamma$, and\footnote{Just take $Z=q(X)$ for $q=\sum_{j=1}^n z_j 1_{P_j}$, where $z_j,P_j$ are given by the identity $\zeta=\sum_{j=1}^n \gamma(P_j) \delta_{z_j} $.} $Z$ which  is a $n$-quantisation of $G$ and whose law  is $\zeta$.
\end{enumerate} 
\end{remark}

\subsection{Barycentric quantisation}
The  quantisation $\zeta=\sum_{j} \gamma(P_j) \delta_{z_j} $ of $\gamma$  satisfies $\zeta \leqcx \gamma$ if  $z_j$ is  the barycentre  of the probability $\gamma_{j} \coloneqq  \frac{\mathbbm{1}_{P_{j}}}{\gamma\pp{P_{j}}} \cdot \,\gamma$, for each $j$. 
Indeed, Jensen inequality states that  $\delta_{z_j} \leqcx \gamma_{j}$, from which  $\zeta \leqcx \gamma$ follows  trivially by linearity of the integral.
This leads us to the following definition.

\begin{definition}
We will say that $\zeta$ is the \emph{proper barycentric quantisation} of $\gamma \in \cP_1(B)$ with respect to  $\Pi\coloneqq (P_j)_{j=1}^n$ if $\Pi\subseteq \cB(B)$ is a partition of $B$ and $\zeta$ is of the form 
\begin{equation}
\label{eq: proper barycentric quantisation of gamma}
	\zeta  = \sum_{j} \gamma(P_j) \delta_{z_j}, \quad \text{ with } z_j\coloneqq\frac{1}{\gamma(P_j)}\int_{P_j} y \gamma(dy) ,
\end{equation} 
where the sum  runs over those values of $j=1, \ldots, n$ for which $\gamma(P_j)>0$ (and for such values $z_j$ is defined). Notice that in this case $\zeta \in \cP_1(B)$, since $\zeta$ is finitely-supported.
\end{definition}

Consider a probability measure $\gamma\in\cP_1(B)$ and the set 
$$
\cP^{(n)}_{\text{cx}}\pp{\gamma} \coloneqq 
\gp{\zeta \in \cP_1(B): \abs{\text{supp}\pp{\zeta}} \leq n\in\bN,\ \zeta \leqcx\gamma} ,
$$
of probability measures, supported on at most $n\in\bN$ points,  which are smaller than $\gamma$ in convex order. The following example shows that a generic element of $\cP^{(n)}_{\text{cx}}\pp{\gamma}$ cannot be obtained as a quantisation of $\gamma$.
\begin{example}
\label{Example: counterexample barycentric quantisation}
	Consider the following probability measures on $\cB(\bR)$
\begin{align*}
	\mu(dx) &= 
	\frac{1}{4}\delta_{(-1)}(dx)+\frac{1}{2}\delta_{(0)}(dx)+\frac{1}{4}\delta_{(1)}(dx) ,\\
	\nu(dy) &= 
	\frac{1}{2}\pp{\delta_{(-1)}(dy) + \delta_{(1)}(dy)} .
\end{align*}
and the kernel $\kappa_x(dy)$ defined by
$$
\kappa_x(dy) =
\begin{cases}
	\delta_{(-1)}(dy) &
	\text{ if } x = -1\\
	\frac{1}{2}\pp{\delta_{(-1)}(dy) + \delta_{(1)}(dy)} &
	\text{ if } x = 0\\	
	\delta_{(1)}(dy) & 
	\text{ if } x = 1 .
\end{cases}
$$
Since  $\nu$ can be obtained from $\mu$ by splitting the mass concentrated at  $x=0$ and sending it in equal parts to the points $x=\pm1$ (and not moving the mass concentrated at $x=\pm1$) we have that $\nu(dy) =\int_\bR \mu(dx)\kappa_x(dy)$. 
Since the measure $\kappa_x(dy)$ has barycentre $x$ for all $x$, and Jensen inequality gives $\delta_x\leqcx \kappa_x$, integrating in $\mu(dx)$ we get that  $\mu \leqcx \nu$ (by linearity of the integral), and  so $\mu \in\cP^{(3)}_{\text{cx}}\pp{\nu}$.
However, $\mu$ can not be obtained as a quantisation of $\nu$, since $\mu$ is supported on strictly more points than $\nu$.

\end{example}

Luckily, the above example  suggests a construction which allows to obtain all of $\cP^{(n)}_{\text{cx}}\pp{\gamma}$. Indeed, notice that if 
$$\pi(dx,dy)\coloneqq \mu(dx)\kappa_x(dy), \quad P_{1}=(-\infty,0),\quad P_2=\{0\}, \quad P_3=(0,\infty),$$
then $\mu=\sum_{j=1}^3 \pi_j \delta_{y_j}$, where $\pi_j\coloneqq \pi(P_j \times \bR) =\mu(P_j)$ and $y_1=\{-1\},y_2=\{0\},y_3=\{1\}$ can be expressed as 
$$y_j\coloneqq \frac{1}{\pi_{j}}\int_{P_{j}\times \bR }y\,\pi\pp{dx, dy}, \quad  j=1,2,3.$$
This leads us to the following definition.
\begin{definition}
\label{Definition: barycentric quantisation measures}
Given  $\gamma, \zeta \in \cP_1(B)$, with $\zeta$ finitely-supported (and thus of the form $\zeta  = \sum_{i=1}^n w_i \delta_{z_i}$), we will call $\zeta$ the  \emph{barycentric quantisation}  of $\gamma$ with respect to $\pp{\pi, \Pi_1, \Pi_2}$ if  $\pi$ is a probability on $B\times B$ whose second marginal is $\gamma$ and  $\zeta$  is of the form 
\begin{equation}
\label{eq: barycentric quantisation of gamma}
	\zeta = \sum_{i\in I, j\in J: \pi_{i, j}>0}\pi_{i, j}\delta_{y_{i, j}} ,
\end{equation} 
where $\pi_{i, j}\coloneqq\pi\pp{P_{1,i}\times P_{2,j}}$ for all $i\in i, j\in J$, and 
\begin{align}
\label{eq: y bar pi}
y_{i, j} \coloneqq 
\frac{1}{\pi_{i, j}}\int_{P_{1,i} \times P_{2,j} }y\,\pi\pp{dx, dy} , \text{ for } i\in I ,j \in J: \pi_{i, j}>0,
\end{align}
for some finite partitions $\Pi_1\coloneqq(P_{1,i})_{i\in I}, \Pi_2\coloneqq(P_{2,j})_{j \in J} \subseteq \cB(B)$ of $B$. 
\end{definition}
As will we see in Remark \ref{rem: bar q generalises proper bar q},  the concept of  barycentric quantisation generalises that of proper  barycentric quantisation. The two concepts are not equivalent since, as we saw, $\mu$ in Example \ref{Example: counterexample barycentric quantisation} is a barycentric quantisation of $\nu$, but it is not a proper barycentric quantisation of $\nu$.

In Remark \ref{re: quantisation of rv and measures} we saw that the notion of quantisation for measures admits a corresponding notion of quantisation for random variables. As we will now see, the same is true for the barycentric quantisation; this link is worth developing because, even if ultimately we are interested in the measures themselves, in proofs it is often easier to work with random variables. This is of course the same reason why Strassen's Theorem, and Skorokhod's representation Theorem, are so useful. 

We will need  the following simple  result.
\begin{lemma}
\label{le: auxiliary lemma}
$H \subseteq \cF$ is a finite partition of $\Omega$ if and only if $H=W^{-1}(\Pi)$ for some finite partition $\Pi$ of $B$ and $B$-valued random variable $W$. It is always possible to choose $\Pi, W$ such that $W$ takes exactly one value in each element of $\Pi$; in this case $H$ and $\Pi$ have the same number of elements, and $H$ is the family of sets of the form $\{W=w\}$, where $w$ ranges over the image of $W$.

If $G$ is a $B$-valued random variable which takes finitely many values and $H \subseteq \cF$ a finite partition of $\Omega$, then $H \subseteq \sigma(G)$ if and only if $H=G^{-1}(\Pi)$ for some finite partition $\Pi \subseteq \cB(B)$ of $B$.
\end{lemma}
\begin{proof}
Trivially, if $\Pi$ is a finite partition of $B$ and $W$ a $B$-valued random variable then $H\coloneqq W^{-1}(\Pi)$ is a partition of $\Omega$, and if $W$ takes exactly one value in each element of $\Pi$ then  $H=\{\{W=w\}: w\in Im(W)\}$, and so $H$ has as many elements as $\Pi$.

Conversely, given a partition $H=\{H_i\}_{i=1}^n \subseteq \cF$ of $\Omega$, one can build  as follows a finite partition $\Pi$ of $B$ and a $W$ which takes exactly one value in each element of $\Pi$ and is such that $H=W^{-1}(\Pi)$: choose distinct points $b_1, \ldots, b_n\in B$ and a partition  $\Pi\coloneqq\{P_{i}\}_{i=1}^n \subseteq \cB(B)$ of $B$ such that $b_i\in P_{i}$ for all $i$ (for example one can take $P_{i}=\{b_i\}$ for $i<n$, and $P_{n}=B\setminus \{b_1, \ldots, b_{n-1}\}$), and define $W\coloneqq\sum_{i=1}^n b_i 1_{G_i}$. 

Now  consider the second statement: let $G$ takes finitely-many values and $H \subseteq \cF$ be a finite partition. If $H=G^{-1}(\Pi)$ for some finite partition $\Pi$ of $B$ then   $H \subseteq \sigma(G)$ follows from $\Pi \subseteq \cB(B)$.  Conversely, if $H \subseteq \sigma(G)$ then each element $H_i$ of $H=\{H_i\}_{i=1}^n$ is a union of atoms of $\sigma(G)$, and thus is of the form $\cup_{g\in S_i}\{G=g\}= \{G\in S_i\}$ for some $S_i \subseteq Im(G)$. Since $H_i \cap H_j=\emptyset$ for $i\neq j$, if $P_i\coloneqq S_i \setminus \cup_{j<i} S_j$ for $i<n$ and $P_n=B \setminus \cup_{i<n} P_i$ then $\{G\in S_i\}= \{G\in P_i\}$,  and so  $\Pi\coloneqq \{P_i\}_{i=1}^n$ is a partition of $B$ such that $H=G^{-1}(\Pi)$.
\end{proof}
 
\begin{definition}
\label{Definition: barycentric quantisation rvs}
Given $Z,G\in L^1\pp{\bP, B}$, $n\in \bN$, we will call $Z$ a \emph{barycentric quantisation} of $G$ with respect to $\cG$ if $Z$ is of the form 
\begin{equation}
\label{eq: barycentric quantisation of G}
	Z=\bE\sq{G\vert\cG} ,
\end{equation}
for some finite $\sigma$-algebra $\cG\subseteq \cF$. We will say that a  barycentric quantisation  $Z$ of $G$ with respect to $\cG$ is \emph{proper} if $\cG \subseteq \sigma(G)$.  
\end{definition} 

\begin{remark}
\label{re: equiv bar quant rv}
If $Z$ a barycentric quantisation of $G$ with respect to $\cG$ then it is also a barycentric quantisation of $G$ with respect to $\sigma(Z)$ since
$$Z=  \bE[Z \vert \sigma(Z)]= \bE[\bE[G\vert \cG ] \vert \sigma(Z)]= \bE[G \vert \sigma(Z)],$$ 
Since $Z$ takes\footnote{Of course $Z$ is an equivalence class, so by `$Z$ takes at most $n$ values' we mean that it admits a representative which takes at most $n$ values, or equivalently that the   law of $Z$ is supported on at most $n$ points.}
at most as many values as the number of atoms of $\cG$ (because $Z=\bE[G\vert \cG ]$ is constant on each atom of $\cG$), $\sigma(Z)$  has at most as many  atoms as $\cG$ (since the atoms of $\sigma(Z)$ are the sets of the form $\{Z=z\}$, for $z$ in the support of $\cL(Z)$).
\end{remark} 
The two definitions of barycentric quantisation (for probabilities on $B$, and for $B$-valued random variables) are strictly connected, as we now explain  first considering a general barycentric quantisation, and then considering a proper one. 

\begin{remark}
\label{re: quant for rv vs for measures}
Let $Z$ be the barycentric quantisation of $G$ with respect to the  finite $\sigma$-algebra $\cG\subseteq \cF$.
If $H$ is the set of atoms of $\cG$, Lemma \ref{le: auxiliary lemma} gives that $H=W^{-1}(\Pi)$ for some $\Pi$ and  $W$, so $\cG=\sigma(W^{-1}(\Pi))$, and so by Lemma \ref{le: law of Z} the law $\zeta$ of $Z$ is a  barycentric quantisation of the law $\gamma$ of $G$ with respect to $(\pi, \Pi,\{B\})$, where $\pi$ is the law of $(W,G)$.
Moreover, when we applied Lemma \ref{le: auxiliary lemma} to get $\Pi$ and $W$, we could have chosen to have that $W$ takes one value in each element of $\Pi$ and thus $\Pi$ has as many element as the atoms of $\cG$.

Conversely, let $\zeta$ be a barycentric quantisation of $\gamma$ with respect to $\pp{\pi,\Pi_1,\Pi_2}$. 
If $X,G$ are the projections on the first and second coordinates, defined on the probability space $B \times B$  endowed with the probability $\pi$, and $Z\coloneqq \bE[G|\cG]$ where $\cG\coloneqq\sigma( \{P_{1,i} \times P_{2,j}\}_{i, j})$, then trivially $G$ has law $\gamma$ and $Z$ is a barycentric quantisation of $G$ with respect to $\cG$; moreover  $Z$ has law $\zeta$, since 
$$Z=\sum_{i,j} 1_{P_{i,j}} \frac{\bE[1_{P_{1,i}\times P_{2,j}}G]}{\bP(P_{1,i}\times P_{2,j})}=\sum_{i,j} 1_{P_{i,j}} \frac{\bE[1_{P_{1,i}}(X) 1_{P_{2,j}}(G) G]}{\bP(P_{1,i}\times P_{2,j})}=\sum_{i,j} 1_{P_{i,j}} y_{i,j},$$
where $y_{i,j}$ are as in (\ref{eq: y bar pi}).
\end{remark} 

\begin{remark}
\label{re: proper quant for rv vs for measures}
If $Z$ is a \emph{proper} barycentric quantisation of $G$ with respect to $\cG$ then, calling $H$ the family of atoms of $\cG$, we can write  $\cG=\sigma(G^{-1}(\Pi))$ for some partition $\Pi\subseteq \cB(B)$ of $B$ using  Lemma \ref{le: auxiliary lemma}, and then the law $\zeta$ of $Z$  is a  barycentric quantisation of the law $\gamma$ of $G$  with respect to  $\Pi$ by Lemma \ref{le: law of Z}.

Conversely, if $\zeta$ is a \emph{proper} barycentric quantisation of $\gamma$ with respect to the partition $\Pi\subseteq \cB(B)$ of $B$, and $G$ has  law $\gamma$, then $Z\coloneqq \bE[G|\sigma(G^{-1}(\Pi))]$ is  a \emph{proper} barycentric quantisation of $G$ with respect to $\cG\coloneqq \sigma(G^{-1}(\Pi))$ since $\sigma(G^{-1}(\Pi))\subseteq \sigma(G)$, and $Z$ has law  $\zeta$ by Lemma \ref{le: law of Z}. 
\end{remark} 

\begin{remark}
\label{thm: estimates speed quant bar}
Remarks \ref{re: quant for rv vs for measures} and \ref{re: proper quant for rv vs for measures} imply that the estimates of the speed of convergence provided in Theorem \ref{th: speed of convergence} hold whenever $\hat{\nu}/\hat{Y}$ is a barycentric quantisation (resp.~$\hat{\mu}/\hat{X}$ is a proper barycentric quantisation) of $\nu/Y$ (resp.~of $\mu/X$). Since any \emph{proper} barycentric  quantisation is a barycentric  quantisation, the last lines of the proof of Theorem \ref{th: speed of convergence} should now appear obvious. 
\end{remark}

\begin{theorem}
\label{prop: alternative definition barycentric quantisation}
If $\zeta,\gamma \in \cP_1(B)$ then t.f.a.e.:
\begin{enumerate}
\item \label{it: cx} $\zeta \leqcx \gamma$ and $\zeta$ is finitely-supported
\item \label{it: generic bq} $\zeta$ is a barycentric quantisation of $\gamma$ with respect to $\pp{\pi,\Pi_1,\Pi_2}$, for some $\pi\in \cP_1(B)$ with second marginal $\gamma$ and $\Pi_1,\Pi_2\subseteq \cB(B)$ finite partitions of $B$.
\item \label{it:  bq trivial Pi_1} $\zeta$ is a barycentric quantisation of $\gamma$ with respect to $\pp{\pi,\Pi_1,\{B\}}$, for some $\pi\in \cP_1(B)$ with second marginal $\gamma$ and $\Pi_1\subseteq \cB(B)$ finite partition of $B$.
\end{enumerate} 
\end{theorem} 
\begin{proof}
Trivially, item \ref{it:  bq trivial Pi_1}  implies  item \ref{it: generic bq}. If item \ref{it: generic bq} holds, by item \ref{re: quant for rv vs for measures} there exist $Z\sim \zeta, G \sim \gamma$ and a finite $\sigma$-algebra $\cG\subseteq \cF$   such that $\bE[G|\cG]=Z$, and so item \ref{it: cx} follows from the easy half of Strassen's theorem (Theorem \ref{th: Strassen}). Finally, if item \ref{it: cx} holds then the hard half of Strassen's theorem implies the existence of $Z\sim \zeta, G \sim \gamma$ such that $\bE[G|Z]=Z$; since $\zeta$ is finitely supported, $\cG\coloneqq \sigma(Z)$ is finite and so by Remark \ref{re: quant for rv vs for measures} $\zeta$ is a barycentric quantisation of $\gamma$ with respect to $(\pi,\Pi_1,\{B\})$ for some $\pi,\Pi_1$.
\end{proof} 

\begin{remark}
\label{rem: bar q generalises proper bar q}
	Notice that the Proposition \ref{prop: alternative definition barycentric quantisation} shows that, in the definition of barycentric quantisation of $\gamma$, we could equivalently have demanded that the first marginal of $\pi$ has finite support and $\Pi_2=\{B\}$. However, doing this is not a good idea. Indeed, 	allowing for a more general $\Pi_2$ allows to more easily to consider settings which naturally involve two partitions, for example making it clear that if $\zeta$ is a \emph{proper}  barycentric quantisation of $\gamma$ then $\zeta$ is a  barycentric quantisation of $\gamma$; moreover, the choice of $\Pi_2$ strongly affect the upper bounds we discussed in Theorem \ref{thm: estimates speed quant bar}, so in this regard $\Pi_2$ should rather be chosen to have sets as small (in diameter) as possible, instead of being composed of only one set $B$ (of infinite diameter).
\end{remark}

The next proposition shows that the situation of Example \ref{Example: counterexample barycentric quantisation} is standard, and makes Proposition \ref{prop: alternative definition barycentric quantisation} slightly more precise, as it keeps track of the number of points in the support of a finitely-supported measure.
\begin{proposition}
\label{Proposition: barycentric quantisation span set fint supp measures smaller in convex order}
The set $\cP^{(n)}_{\text{cx}}\pp{\gamma}$ coincides with the set of barycentric quantisations of $\gamma$ with respect to $(\pi,\Pi_1,\{B\})$, where $\Pi_1$ spans all partitions of $B$ made of at most $n$ elements and $\pi$ spans all probabilities on $B \times B$ whose second marginal is $\gamma$.
\end{proposition}
\begin{proof}
We first show that the stated barycentric quantisations are in $\cP^{(n)}_{\text{cx}}\pp{\gamma}$; then we prove the converse, i.e.~that every $\zeta\in \cP^{(n)}_{\text{cx}}\pp{\gamma}$ comes from a barycentric quantisation of the kind considered in the statement.

If $\zeta$ is a barycentric quantisations of $\gamma$ with respect to $(\pi,\Pi_1,\{B\})$, by definition the number of points in the support of $\zeta$ equals the number of $P_i\in \Pi_1$ such that $\pi(P_i \times B)>0$; in particular, if $\Pi_1$ has at most $n$ elements then  $\zeta$ is supported on at most $n$ points. Since Proposition \ref{prop: alternative definition barycentric quantisation} shows that $\zeta\leqcx \gamma$, we proved $\zeta\in \cP^{(n)}_{\text{cx}}\pp{\gamma}$. 

Conversely, if $\zeta \in \cP^{(n)}_{\text{cx}}\pp{\gamma}$, by Strassen's Theorem (i.e.~Theorem \ref{th: Strassen})   there exists  $G, Z$ with laws  $\gamma, \zeta$ and such that $\bE\sq{G\vert Z} = Z$, so $Z$ is a barycentric quantisation of $G$ with respect to $\cG\coloneqq \sigma(Z)$. 
If $\{z_i\}_{i=1}^k$ is the support of $\zeta$ (so that $k\leq n$), the set of atoms $\{Z=z_i\}_{i=1}^k$ of $\cG$  equals $Z^{-1}(\Pi_1)$, where $\Pi_1=(P_i)_{i=1}^k$ is the partition of $B$ given by $P_i\coloneqq \{z_i\}$ for $i<k$, and $P_k\coloneqq B \setminus \cup_{i<k} P_i$.  Thus   by Lemma \ref{le: law of Z} the law $\zeta$ of $Z$ is a  barycentric quantisation of the law $\gamma$ of $G$ with respect to $(\pi, \Pi_1,\{B\})$, where $\pi$ is the law of $(Z,G)$.
Since $\pi$ has second marginal $\gamma$ and $\Pi_1$ has $k\leq n$ elements, the proof is concluded.
\end{proof}
Given the link between barycentric quantisations for measures and for random variables expressed in Remark \ref{re: quant for rv vs for measures}, it is natural to expect a version of Proposition \ref{prop: alternative definition barycentric quantisation} for random variables. The reason why such version is not immediately obvious, is that we have so far not seen a  definition of convex order for random variables analogous to that for measures, so we now introduce it.
\begin{definition}
\label{def: convex order for rv}
Given $Z,G\in L^1(B,\bP)$, we say that $Z \leqcx G$ if there exists a $\sigma$-algebra $\cG \subseteq \cF$ such that  $f(Z)\leq \bE[f(G)|\cG]$ holds for every  $f:B\to \bR$ Lipschitz and convex. 
\end{definition} 
The following simple fact is the analogue for random variables of the fact that $\mu \leqcx \nu$ if and only if there exists a kernel $K$ such that $\nu=\mu K$ and $\text{bar}(K_x)=x$ for $\mu$ a.e.~$x$.
\begin{lemma}
\label{thm: bar quant rv span all rv smaller in conver order}
Given $Z,G\in L^1(B,\bP)$, t.f.a.e.:
\begin{enumerate}
\item \label{it: ineq} $Z \leqcx G$
\item \label{it: ce} $Z=\bE[G |\cG]$ for some $\sigma$-algebra $\cG \subseteq \cF$
\end{enumerate} 
If the above conditions hold, then $f(Z)\leq \bE[f(G)|Z]$ holds for every  $f:B\to \bR$ Lipschitz and convex, and in particular $Z=\bE[G |Z]$.
\end{lemma}
\begin{proof}
Assume item \ref{it: ce} holds. The conditional Jensen inequality implies item \ref{it: ineq}, with the same $\cG$; since the tower property of conditional expectation shows that $Z=\bE[G |Z]$,  one can take w.l.o.g. $\cG=\sigma(Z)$ also in (the definition used in) item \ref{it: ineq}. Conversely, applying the inequality $f(Z)\leq \bE[f(G)|\cG]$ to $f(x)=x$ and to $f(x)=-x$, gives $Z=\bE[G |\cG]$.
\end{proof}  
Analogously we can state the equivalent of  Proposition \ref{prop: alternative definition barycentric quantisation} for random variables.
\begin{corollary}
\label{co: bar quant rv}
Given $Z,G\in L^1(B,\bP)$, t.f.a.e.:
\begin{enumerate}
\item \label{it: fin ineq}  $Z \leqcx G$ and $Z$ only takes finitely many values
\item \label{it: fin ce} $Z$ is barycentric quantisation of $G$
\end{enumerate}
\end{corollary} 
\begin{proof} If item \ref{it: fin ineq} holds then Theorem \ref{thm: bar quant rv span all rv smaller in conver order} shows that $Z=\bE[G |Z]$; since $\sigma(Z)$ takes finitely many values,  item \ref{it: fin ce}  holds.
 Conversely, if item \ref{it: fin ce} holds then $Z$ takes finitely many values (since it is $\cG$-measurable), and  Theorem \ref{thm: bar quant rv span all rv smaller in conver order}
   implies  $Z \leqcx G$.
\end{proof}

%%%%%%%%%%%%%%%%%%%%%%%%%%%%%%%%%%%%%%%%%%%%%%
%% Appendix:                                %%
%%%%%%%%%%%%%%%%%%%%%%%%%%%%%%%%%%%%%%%%%%%%%%
\begin{appendix}
\section{Discrete MOT - primal LP}
\label{Section: Appendix - algo primal LP}
Consider the MOT problem
$$
\inf_{\pi\in\cM\pp{\mu, \nu}}\bE\sq{c(X, Y)}
$$
For a given $n\in \bN$, a (sub-optimal) martingale transport $\tilde{\pi} \in \cM(\mu, \nu)$ and two finite partitions $\Pi^n_1,\ \Pi^n_2$ of $B$, we can apply the proposed martingale quantisation scheme so that to reduce the above MOT to the following LP
\begin{align}
\tag{MOT-LP}
\label{eq: Discrete primal MOT}
	\begin{split}
		\min_{p\in \bR^{n_\mu\times n_\nu}_{++}} &
		\sum_{i=1}^{n_\mu}\sum_{j=1}^{n_\nu} p_{i, j}c_{i, j} \\
		\sum_{j}p_{i, j} & = \omega^{(n_\mu)}_{X,i} \quad i=1,\dots,n_\mu \\
		\sum_{i}p_{i, j} & = \omega^{(n_\nu)}_{X,i} \quad j=1,\dots,n_\nu \\
		\sum_{j}p_{i, j} & y^{(n_\nu)}_j  = \omega^{(n_\mu)}_{X,i} x^{(n_\mu)}_i \quad i=1,\dots,n_\mu
	\end{split}
\end{align}
where that $n_\mu \coloneqq \#$supp$\pp{\mu_n} \leq \#\Pi^n_1,\ n_\nu \coloneqq\#$supp$\pp{\nu_n}\leq \#\Pi^n_1 \cdot \#\Pi^n_2$.

Problem (\ref{eq: Discrete primal MOT}) can be readily implemented once recasted in the canonical form
\begin{align}
\label{eq: LP canonical form}
	\begin{split}
		 \min_{x} &\; d^T x\\
		 Ax  & =b \\
		 x &\geq 0
	 \end{split}
\end{align}
where $d$ and $x$ are (column) vectors (whereas $p$ and $c$ are matrices). To do so, it is sufficient to define $x\coloneqq$ vec$\pp{p}$ and  $d\coloneqq$ vec$\pp{c}$, i.e.~the vectorisations of the transport $p$ and the cost matrix $c$. The constraint matrix is given by
$
A\coloneqq 
\begin{bmatrix}
	\text{A}_{\cC_1},\
	\text{A}_{\cC_2},\
	\text{A}_{\cC_3}
\end{bmatrix}^T
$
where the matrices 
$\text{A}_{\cC_1}$,
$\text{A}_{\cC_2}$ and
$\text{A}_{\cC_3}$
(whose dimensions are:
$n_\mu \times n_\mu n_\nu, \,
n_\nu \times n_\mu n_\nu \,$ and
$n_\mu \times n_\mu n_\nu$, resp.)
are defined by

\newcommand\undermat[2]{%
  \makebox[0pt][l]{$\smash[b]{\underbrace{
  \phantom{
    \begin{array}{cccc}
	00&00&00&00
    \end{array}
  }
  }_{\text{$#1$}}}$}
    #2    	
    }
\newcommand\undermatshort[2]{%
  \makebox[0pt][l]{$\smash[b]{\underbrace{
  \phantom{
    \begin{array}{ccc}
	00&0&00
    \end{array}
  }
  }_{\text{$#1$}}}$}
    #2    	
    }
\newcommand\undermatlong[2]{%
  \makebox[0pt][l]{$\smash[b]{\underbrace{
  \phantom{
    \begin{array}{cccc}
	0000&000&00&0000
    \end{array}
  }
  }_{\text{$#1$}}}$}
    #2    	
    }
    
{\small{
\begin{align*}
\text{A}_{\cC_1} &\coloneqq
\left[
\begin{array}{cccccccccc}
\undermat{n_\nu}{1 & 0 & $\dots$ & 0} &
\undermat{n_\nu}{1 & 0 & $\dots$ & 0} &
\dots & \dots \\
\\
\\
0 & \undermat{n_\nu}{1 & 0 & $\dots$ & 0} &
\undermat{n_\nu}{1 & 0 & $\dots$ & 0} &
\dots\\
\\
\\
\dots &
\dots &
\dots &
\dots &
\dots &
\dots &
\dots &
\dots &
\dots &
\dots
\\
\\
\undermatshort{n_\nu -1}{0 & $\dots$ & 0} & 
\undermat{n_\nu}{1 & 0 & $\dots$ & 0} &
\dots & 
\dots & 1
\end{array}
\right],
\\
\\
\text{A}_{\cC_2} &\coloneqq
\left[
\begin{array}{cccccccccc}
\undermat{n_\mu}{1 & $\dots$ & $\dots$ & 1} &
0 & 
0 & 
\dots & 
\dots &
\dots & 0
\\
\\
0 & \undermat{n_\mu}{1 & $\dots$ & $\dots$ & 1} &
0 & 
0 & 
\dots &
\dots & 0
\\
\\
\dots &
\dots &
\dots &
\dots &
\dots &
\dots &
\dots &
\dots &
\dots &
\dots
\\
\\
0 &
0 &
\dots &
\dots &
\dots &
0 &
\undermat{n_\mu}{1 & $\dots$ & $\dots$ & 1}
\end{array}
\right],
\\
\\
\text{A}_{\cC_3} &\coloneqq
\left[
\begin{array}{ccccccccccccc}
\undermatlong{n_\nu}{y^{(n_\nu)}_1 & 0 & $\dots$ & 0} &
\undermatlong{n_\nu}{y^{(n_\nu)}_2 & 0 & $\dots$ & 0} &
\dots & 
\undermatlong{n_\nu}{y^{(n_\nu)}_{n_\nu} & 0 & $\dots$ & 0}
\\
\\
\\
\undermatlong{n_\nu}{0 & y^{(n_\nu)}_1  & $\dots$ & 0} &
\undermatlong{n_\nu}{0 & y^{(n_\nu)}_2  & $\dots$ & 0} &
\dots & 
\undermatlong{n_\nu}{0 & y^{(n_\nu)}_{n_\nu} & $\dots$ & 0}
\\
\\
\\
\dots &
\dots &
\dots &
\dots &
\dots &
\dots &
\dots &
\dots &
\dots &
\dots &
\dots &
\dots &
\dots 
\\
\\
\undermatlong{n_\nu}{0 & 0 & $\dots$ & y^{(n_\nu)}_1} &
\undermatlong{n_\nu}{0 & 0 & $\dots$ & y^{(n_\nu)}_2} &
\dots & 
\undermatlong{n_\nu}{0 & $\dots$ & 0 & y^{(n_\nu)}_{n_\nu}}
\end{array}
\right],
\end{align*}
}}
\\
\noindent whereas 
$b \coloneqq
\begin{bmatrix}
	\omega^{(n_\mu)}_X,\ 
	\omega^{(n_\nu)}_Y,\ 
	\omega^{(n_\mu)}_X \cdot x^{(n_\mu)}
\end{bmatrix}^T
$.

\begin{remark}
	Problem (\ref{eq: Discrete primal MOT}) recasted in canonical form has $n_\mu n_\nu$ variables and $2n_\mu+n_\nu$ constraints.
\end{remark}
\end{appendix}

\bibliography{Quantise_convex_order}{}
\bibliographystyle{abbrv}

\end{document}